\documentclass[12pt,leqno]{amsart}
\usepackage{amssymb,amsfonts,amsmath,amsopn,amstext,amscd,extarrows,latexsym,xy,hyperref,mathrsfs,verbatim}
\usepackage{epsfig}
\input xy
\xyoption{all}

\allowdisplaybreaks[1]

\theoremstyle{remark}
\newtheorem{para}{\bf}[section]

\newtheorem{assumption}[para]{\bf Assumption}

\theoremstyle{definition}
\newtheorem{exam}[para]{\bf Example}

\newtheorem{dfn}[para]{\bf Definition}

\theoremstyle{plain}
\newtheorem{thm}[para]{\bf Theorem}
\newtheorem{lemma}[para]{\bf Lemma}
\newtheorem{sublemma}[para]{\bf Sublemma}
\newtheorem{cor}[para]{\bf Corollary}
\newtheorem{prop}[para]{\bf Proposition}

\newenvironment{numequation}
{\addtocounter{enumi}{1}\begin{equation}}{\end{equation}}

\newcommand{\bB}{{\bf B}}

\newcommand{\bG}{{\bf G}}

\newcommand{\bL}{{\bf L}}
\newcommand{\bM}{{\bf M}}

\newcommand{\bP}{{\bf P}}

\newcommand{\bT}{{\bf T}}
\newcommand{\bU}{{\bf U}}

\newcommand{\bbQ}{{\mathbb Q}}
\newcommand{\bbR}{{\mathbb R}}

\newcommand{\bbZ}{{\mathbb Z}}

\newcommand{\cF}{{\mathcal F}}

\newcommand{\cH}{{\mathcal H}}
\newcommand{\cI}{{\mathcal I}}
\newcommand{\cJ}{{\mathcal J}}

\newcommand{\cO}{{\mathcal O}}

\newcommand{\frb}{{\mathfrak b}}

\newcommand{\frd}{{\mathfrak d}}

\newcommand{\frg}{{\mathfrak g}}

\newcommand{\frl}{{\mathfrak l}}
\renewcommand{\frm}{{\mathfrak m}}

\newcommand{\frp}{{\mathfrak p}}
\newcommand{\frq}{{\mathfrak q}}

\newcommand{\frt}{{\mathfrak t}}
\newcommand{\fru}{{\mathfrak u}}

\newcommand{\frx}{{\mathfrak x}}

\newcommand{\frz}{{\mathfrak z}}

\newcommand{\tH}{\widetilde{H}}

\newcommand{\uM}{{\underline M}}
\newcommand{\uN}{{\underline N}}

\newcommand{\vep}{\varepsilon}

\newcommand{\vphi}{\varphi}

\newcommand{\Z}{{\mathbb Z}}

\newcommand{\ad}{{\rm ad}}
\newcommand{\Ad}{{\rm Ad}}

\newcommand{\bksl}{\backslash}

\newcommand{\End}{{\rm End}}

\newcommand{\GL}{{\rm GL}}
\newcommand{\Hom}{{\rm Hom}}
\newcommand{\hra}{\hookrightarrow}

\newcommand{\ind}{{\rm ind}}
\newcommand{\Ind}{{\rm Ind}}
\newcommand{\Lie}{{\rm Lie}}
\newcommand{\lra}{\longrightarrow}
\newcommand{\midc}{\;|\;}

\newcommand{\Qp}{{\bbQ}_p}
\newcommand{\ra}{\rightarrow}
\newcommand{\Rep}{{\rm Rep}}

\newcommand{\sub}{\subset}
\newcommand{\supp}{{\rm supp}}
\newcommand{\alg}{{\rm alg}}

\newcommand\Zp{{{\bbZ}_p}}

\newcommand{\Pf}{{\it Proof. }}
\renewcommand{\qed}{{\hfill{\space} $\Box$}}

\begin{document}

\title{Category $\cO$ and locally analytic representations}
\author{Sascha Orlik}
\address{Fachbereich C - Mathematik und Naturwissenschaften, Bergische Universit\"at Wuppertal,
Gau{\ss}stra\ss{}e 20, D-42119 Wuppertal, Germany}
\email{orlik@math.uni-wuppertal.de}
\author{Matthias Strauch}
\address{Indiana University, Department of Mathematics, Rawles Hall, Bloomington, IN 47405, USA}
\email{mstrauch@indiana.edu}
\thanks{M. S. would like to acknowledge the support of the National Science Foundation (award number DMS-1202303).}

\begin{abstract}
For a split reductive group $G$ over a finite extension $L$ of $\Qp$, and a parabolic subgroup $P \sub G$ we introduce a category $\cO^P$ which is equipped with a forgetful functor to the parabolic category $\cO^\frp$ of Bernstein, Gelfand and Gelfand. There is a canonical fully faithful embedding of a subcategory $\cO^\frp_\alg$ of $\cO^\frp$ into $\cO^P$, which 'splits' the forgetful map. We then introduce functors from the category $\cO^P$ to the category of locally analytic representations, thereby generalizing the authors' previous work where these functors had been defined on the category $\cO^\frp_\alg$. It is shown that these functors are exact, and a criterion for the  irreducibility of a representation in the image of this functor is proved.
\end{abstract}

\maketitle

\tableofcontents

\section{Introduction}

Let $L$ be a finite extension of $\Qp$ and let $G = \bG(L)$ be the group of $L$-valued points of a split connected reductive algebraic group $\bG$ over $L$. In the paper \cite{OS2} we introduced and studied certain bi-functors

$$\cF^G_P: \cO^\frp_\alg \times \Rep^{\infty}_K(L_P) \lra \Rep^{\rm loc. an.}_K(G) \;.$$

\vskip8pt

Here $\cO^\frp_\alg$ is a full subcategory of the parabolic Bernstein-Gelfand-Gelfand category $\cO^\frp$, attached to a standard parabolic 
subalgebra $\frp \sub \frg$. The subcategory $\cO^\frp_\alg$ consists of objects $M$ such that, if we consider the restriction to the Levi subalgebra $\frl_\frp \sub \frp$, the vector space $M$ can be written as a direct sum of irreducible algebraic representations of the corresponding Levi subgroup $\bL_\bP \sub \bG$.

\vskip8pt

This functor is exact in both arguments, and it was shown that $\cF^G_P(M,V)$ is topologically irreducible if and only if the following two conditions are satisfied

\begin{enumerate}

\item $M \in \cO^\frp_\alg$ is simple, and if $\frq \supset \frp$ is the maximal  parabolic subalgebra such that $M$ lies in $\cO^\frq$, then

\item the smooth induction $\ind^{L_Q}_{L_P(L_Q \cap U_P)}(V)$ is irreducible.
\end{enumerate}

\vskip8pt

Here and in the following we denote by standard roman letters the group of $L$-valued points of the corresponding algebraic group, e.g., $L_P = \bL_\bP(L)$.

\vskip8pt

In this paper we extend the functors to a category $\cO^P$ whose objects consist of those pairs $\uM = (M,\tau)$, where $M$ is an object of $\cO^\frp$ and $\tau: P = \bP(L) \ra \End_K(M)^*$ is a locally finite-dimensional locally analytic representation on $M$ which lifts the Lie algebra representation of $\frp$ on $M$. Mapping $\uM = (M,\tau)$ to $M$ is a  forgetful functor $\omega: \cO^P \to \cO^\frp$, from which we deduce that every object in $\cO^P$ has finite length. 

\vskip8pt
For $\uM = (M, \tau)$ in $\cO^P$ the $U(\frg)$-module structure on $M$ extends canonically to a $D(\frg,P)$-module structure, where $D(\frg,P)$ is the smallest $K$-subalgebra of $D(G)$ containing $U(\frg)$ and $D(P)$. Given $\uM \in \cO^P$ as above, we set 

$$\cF^G_P(\uM):=(D(G)\otimes_{D(\frg,P)} M)' \;.$$

\vskip8pt

This is a strongly admissible locally analytic $G$-representation. In fact, it is canonically isomorphic to a closed subrepresentation of the parabolically induced representation $\Ind^G_P(W')$. Here $W \subset M$ is any finite-dimensional $P$-subrepresentation of $M$ which generates $M$ as a $U(\frg)$-module. Then we have $\cF^G_P(\uM) \cong \Ind^G_P(W')^\frd$, where the latter denotes the subspace of vectors annihilated by $\frd = \ker\Big(U(\frg) \otimes_{U(\frp)} W \twoheadrightarrow M\Big)$. 
As in \cite{OS2}, we extend the functor $\cF^G_P$ to a bi-functor

$$\cF^G_P: \cO^P \times \Rep^{\infty}_K(L_P) \lra \Rep^{\rm loc. an.}_K(G) \;.$$

\vskip8pt

Following the arguments presented in \cite{OS2}, we show that $\cF^G_P$ is exact in both arguments. And we remark that the proof of the ``$PQ$-formula'' in \cite[4.9]{OS2}, i.e., $\cF^G_P(\uM,V)=\cF^G_Q(\uM,\ind^ Q_P(V))$ if $\uM \in \cO^Q$ and $Q \supset P$, carries over to the present setting. Recall that the Lie algebra $\frp$ is called {\it maximal} for $M \in \cO$, if $M \in \cO^\frp$ and $M \notin \cO^\frq$ for any parabolic subalgebra $\frq \supsetneq \frp$. Our main result is then the following statement\footnote{where we assume that $p>2$ if the root system of $G$ contains components of type $B$, $C$, or $F_4$, and that $p>3$ when components of type $G_2$ occur.}

\begin{thm}\label{main_result} Let $\uM = (M,\tau) \in \cO^P$ be such that (i) $M$ is a simple $U(\frg)$-module, and (ii)  $\frp$ is maximal for $M$. Let $V$ be an irreducible smooth $L_P$-representation. Then the locally analytic $G$-representation $\cF^G_P(\uM,V)$ is topologically irreducible.
\end{thm}

A special but essential case is when $V$ is the trivial one-dimensional representation. Once one has proved this case, one can use the same arguments as in the proof of \cite[5.8]{OS2}, to deduce the general case. 

\vskip8pt

In order to show the irreducibility of $\cF^G_P(\uM)$, for $\uM$ as in \ref{main_result}, we follow very closely (in some passages word-by-word) the proof of the irreducibility result \cite[5.3]{OS2}. A key role in the proof of \cite[5.3]{OS2}, is played by assertions about ``integrality properties'' of certain relations in simple modules in $\cO^\frp_\alg$, which have been proved in the appendix of \cite{OS2}. But the results \cite[8.10, 8.13]{OS2} of that appendix do not immediately carry over to the present setting, because it has been used there that the eigenvalues of elements in $\frt = \Lie(\bT)$ are ($p$-adically) integral, which is not the case any more for general objects in $\cO$. In section \ref{HWM} we have therefore modified some of the arguments used in \cite[sec. 8]{OS2} to fit the present setting. 

\vskip8pt

We would also like to mention that we are currently working on extending the theory presented here to the case of quasi-split groups.

\vskip8pt

{\it Acknowledgments.} Parts of this paper were written at MSRI in Fall 2014. We are grateful to MSRI and its staff for providing excellent working conditions.

\vskip8pt

{\it Notation and conventions.} We denote by $p$ a prime number and consider fields $L \sub K$ which are both finite extensions of $\Qp$. 
Let $O_L$ and $O_K$ be the rings of integers of $L$, resp. $K$, and let $|\cdot |_K$ be the absolute value on $K$ such that $|p|_K = p^{-1}$. The field $L$ is our ''base field'', whereas we consider $K$ as our ''coefficient field''. For a locally convex $K$-vector space $V$ we denote by $V'_b$ its strong dual, i.e., the $K$-vector space of continuous linear forms equipped with the strong topology of bounded convergence. Sometimes, in particular when $V$ is finite-dimensional, we simplify notation and write $V'$ instead of $V'_b$. All finite-dimensional $K$-vector spaces are equipped with the unique Hausdorff locally convex topology.

\section{The BGG category with $P$-action}

\begin{para} We let $\bG_0$ be a reductive group scheme over $O_L$ and $\bT_0 \sub \bB_0 \sub \bG_0$ a maximal split torus and a Borel subgroup scheme, respectively. We denote by $\bG$, $\bB$, $\bT$ the base change of $\bG_0$, $\bB_0$ and $\bT_0$ to $L$. By $G_0 = \bG_0(O_L)$, $B_0 = \bB_0(O_L)$, etc., and $G = \bG(L)$, $B = \bB(L)$, etc., we denote the corresponding groups of $O_L$-valued points and $L$-valued points, respectively. Standard parabolic subgroups of $\bG$ (resp. $G$) are those which contain $\bB$ (resp. $B$). For each standard parabolic subgroup $\bP$ (or $P$) we let $\bL_\bP$ (or $L_P$) be the unique Levi subgroup which contains $\bT$ (resp. $T$). Finally, Gothic letters $\frg$, $\frp$, etc., will denote the Lie algebras of $\bG$, $\bP$, etc.: $\frg = \Lie(\bG)$, $\frt = \Lie(\bT)$, $\frb = \Lie(\bB)$, $\frp = \Lie(\bP)$, $\frl_P = \Lie(\bL_\bP)$, etc.. Base change to $K$ is usually denoted by the subscript ${}_K$, for instance, $\frg_K = \frg \otimes_L K$. 

\vskip8pt

We make the general convention that we denote by $U(\frg)$, $U(\frp)$, etc., the corresponding
enveloping algebras, {\it after base change to K}, i.e., what would be usually denoted by $U(\frg) \otimes_L K$, $U(\frp) \otimes_L K$, and so on.
\end{para} 

\begin{para} We recall the definition of the BGG-category for the reductive Lie algebra $\frg_K$ and its Borel subalgebra $\frb_K$, as given in \cite[2.5]{OS2}. Let $\frp \supset \frb$ be a standard parabolic subalgebra. The category $\cO^\frp$ is the full subcategory of $U(\frg)$-modules $M$ satisfying the following properties:

\vskip8pt

\begin{enumerate}
\item $M$ is finitely generated as a $U(\frg)$-module. 
\item Viewed as a $\frt_K$-module, $M$ is the direct sum of one-dimensional representions. 
\item The action of $\frp_K$ on $M$ is locally finite, i.e., for every $m \in M$ the $K$-vector space $U(\frp) \cdot m$ is finite-dimensional.
\end{enumerate}

\vskip8pt

In this paper we are interested in objects $M \in \cO^\frp$ which have the property that the representation of $\frp$ integrates to a locally finite-dimensional locally analytic representation of $P = \bP(L)$. Because this is not always possible, we are led to consider the following category.
\end{para}

\vskip8pt

\begin{dfn}
Let $P \subset G$ be a standard parabolic subgroup. We then denote by $\cO^P$ the category whose:

\vskip8pt

- {\it Objects} are pairs $\uM = (M,\tau)$ where $M \in \cO^\frp$ and $\tau: P \to \End_K(M)^*$ is a representation such that the 
following conditions are satisfied: (i) there is an increasing union $M = \bigcup_i M_i$ by finite-dimensional locally analytic $P$-stable 
$K$-subspaces, which are also $\frp$-stable, and such that the derived action of $\frp$ on each $M_i$ (derived from the $P$-action) coincides with the given action coming from the inclusion $\frp \sub \frg$; (ii) for every $p \in P$, every $\frx \in \frg$, and every $m \in M$ one has $\Big(\Ad(p)(\frx)\Big) \cdot m = \tau(p)\Big(\frx \cdot \tau(p^{-1})(m)\Big)$.

\vskip8pt
 
- {\it Morphisms} from $\uM = (M,\tau)$ to $\uM' = (M',\tau')$ are given by $U(\frg)$-module homomorphisms $f: M \ra M'$ such that for all $m \in M$ and all $p \in P$ one has $f(\tau(p)(m))=\tau'(p)(f(m))$.
\end{dfn}

\vskip8pt

By the very definition of this category we have a functor

\begin{eqnarray*}
\omega: \cO^P  & \to & \cO^\frp \;, \\
\uM = (M,\rho) & \mapsto & M \;,
\end{eqnarray*}

\vskip8pt

by forgetting the additional structure of a $P$-representation.

\begin{exam}
(i) In \cite{OS2} we have shown that every object $M$ of the category $\cO^\frp_\alg$ carries a canonical action of $P$. One thus obtains a fully faithful functor $\cO^\frp_\alg \to \cO^P$ whose composition with $\omega$ is just the inclusion $\cO^\frp_\alg \subset \cO^\frp$.
 
\vskip8pt

(ii) The forgetful functor $\cO^P \to \cO^ \frp$ is in general not surjective. For example, consider $\bG = {\rm PGL}_2$ and the standard two-dimensional representation $M$ of $\frg = \frp\frg\frl_2$. This object lies in $\cO^\frg$, but the Lie algebra representation does not integrate to a representation of $G$. Hence $\cO^G \ra \cO^\frg$ is not surjective.
\end{exam}

\vskip8pt

\begin{lemma} 
The category $\cO^P$ is abelian.
\end{lemma}

\begin{proof}
Most of the axioms are easily verified. We show exemplarily that kernels exist. So let 
$f: \uM \to \uN$ be a morphism in $\cO^P$. Then $\ker(f) \in \cO^\frp$.  Let $v \in \ker(f)$. Let $W\subset M$ be a finite-dimensional subspace with $v\in W$ on which $P$ acts locally analytically.
Since $\ker(f)$ is a $P$-representation the same holds true for $W \cap \ker(f).$ As $W$ is finite-dimensional it is automatically a locally analytic $P$-representation. Hence any $v \in \ker(f)$ is contained in a locally analytic $P$-representation. The claim follows.
\end{proof}

\begin{exam}
Let $\uM = (M,\tau)$ be in $\cO^P$ and let $N \sub M$ be a $\frg$-submodule, i.e., $N \in \cO^\frp$. Then it is in general not an object of $\cO^P$.
Indeed let $G = P = \GL_n$ and let $M$ be the two-dimensional representation defined by $g \mapsto (1,(-1)^ {{\rm val}(\det(g))})$. Since this representation is smooth the Lie algebra $\frg$ acts trivially on it. Hence any subspace is a subrepresentation. 
But the subspace $\{(a,a) \midc a \in K\}$ does not lift to a $G$-representation.
\end{exam}

\begin{lemma} 
The category $\cO^P$ is artinian and noetherian.
\end{lemma}

\Pf The category $\cO^\frp$ is noetherian and artinian. Hence the claim follows by applying the forgetful functor 
$\omega$ to any descending or ascending chain of subobjects in $\cO^P.$
\qed

\vskip8pt

\begin{lemma}\label{Lemma_gen_W}
Let $\uM = (M,\tau)$ be an object of $\cO^P$. Then there is a finite-dimensional locally analytic $P$-representation
$W \sub M$ which generates $M$ as a $U(\frg)$-module.
\end{lemma}

\Pf By definition the object $M$ is finitely generated as a $U(\frg)$-module. The finitely many generators are  contained in a locally analytic finite-dimensional $P$-representation $M_i \subset M$. So we may set $W = M_i$. \qed

\vskip8pt

Let $M=M(\lambda)=U(\frg) \otimes_{U(\frb)} K_\lambda \in \cO$ be an ordinary Verma module with respect to a homomorphism $\lambda:\frt \to K$. It is easy to see that any such $\lambda$ can be lifted to a locally analytic character $\tilde{\lambda} : T \to K^*$, i.e., the derivative of $\tilde{\lambda}$ coincides with $\lambda$. The character $\tilde{\lambda}$ is uniquely determined up to a smooth character of $T$. By integrating the action of $\fru$ to an action of $U$ on $M(\lambda)$, cf. \cite[Lemma 3.2]{OS2}, we deduce following

\begin{lemma}\label{locanVerma}
There is a unique object $\uM(\tilde{\lambda})$ in $\cO^B$ with the properties that 

\vskip8pt

(i) $\omega(\uM(\tilde{\lambda})) = M(\lambda)$, and 

\vskip5pt

(ii) $B$ acts on the highest weight vector $1 \otimes 1 \in M(\lambda)$ via the locally analytic character $B \twoheadrightarrow T \stackrel{\tilde{\lambda}}{\lra} K^*$.
\end{lemma}

Consider now two such locally analytic characters $\tilde{\lambda}$, $\tilde{\mu}$. Every $G$-homomorphism $\uM(\tilde{\mu}) \ra \uM(\tilde{\lambda})$ induces a 
$U(\frg)$-homomorphism $\omega(\uM(\tilde{\mu})) = M(\mu) \ra \omega (\uM(\tilde{\lambda})) = M(\lambda)$. Up to scalar there is at most one such map, cf. \cite{H1}. More precisely, one has

$$\dim_K {\rm Hom}_\frg(M(\mu), M(\lambda))=\left\{\begin{array}{cc} 1 & \mu \uparrow \lambda  \\0 & else 
\end{array}\right.$$

\vskip8pt

where $\uparrow$ is the transitive closure of the relation $\mu \uparrow \lambda$ if $\mu=\lambda$ or 
there is some simple reflection $s \in S$ such that $\mu=s\cdot \lambda < \lambda$. 

\vskip8pt

Suppose that $\dim \Hom_\frg(M(\mu), M(\lambda))=1$
and that $\tilde{\lambda}$ is fixed. The difference $\mu - \lambda$ is integral as the highest weight vectors differ by multiplication
with an unipotent element. Hence we may lift $\mu-\lambda$ to an algebraic  homomorphism $\widetilde{\mu-\lambda}: T \to K^*$. Set 

$$\tilde{\mu} := \tilde{\lambda} \cdot \widetilde{\mu-\lambda}$$

\vskip8pt

which is a locally analytic character of $T.$ 
Then we get a non-trivial map $\uM(\tilde{\mu}) \to \uM(\tilde{\lambda})$.

\begin{dfn} Let $\tilde{\lambda}, \tilde{\mu}:T \to K^ \ast$ be two locally analytic characters with derivatives $\lambda$, $\mu$, respectively. We write $\tilde{\mu} \uparrow \tilde{\lambda}$ if and only if  $\mu \uparrow \lambda$ and 
$\tilde{\mu} - \tilde{\lambda} \in X^\ast(T)$ is an  algebraic character.
\end{dfn}

The following result is then immediate.

\begin{prop}
Let $\tilde{\lambda}$, $\tilde{\mu}$ be locally analytic characters of $T$ with derivatives $\lambda$, $\mu$, respectively. Then

$$\dim_K {\rm Hom}_{\cO^B}(\uM(\tilde{\mu}), \uM(\tilde{\lambda})) = \left\{\begin{array}{cc} 1 & \tilde{\mu} \uparrow \tilde{\lambda} \\ \\0 & else 

\end{array}\right. .$$
\end{prop}

\vskip8pt

\setcounter{enumi}{0}

\section{Locally analytic representations}

All distribution algebras appearing in this paper are tacitly assumed to be distribution algebras with coefficient field $K$, and we write $D(H)$ for the distribution algebra $D(H,K)$.

\vskip8pt

\begin{para}{\it Induced representations and their duals.} Let $H'$ be a closed locally $L$-analytic subgroup of a locally $L$-analytic group $H$. Let  $(V,\rho)$ be a locally analytic representation of $H'$ on a $K$-vector space $V$, and consider the induced locally analytic representation

$$\Ind^H_{H'}(V) = \Big\{f \in C^{an}(H,V) \; \midc \; \forall h_1 \in H', \forall h \in H: \; f(hh_1) = \rho(h_1^{-1}) \cdot f(h) \; \Big\} \;.$$

\vskip8pt

The group $H$ acts on this vector space by $(h \cdot f)(x) = f(h^{-1}x)$. There is a canonical map of $D(H,K)$-modules

\vspace{-0.3cm}
\begin{numequation}\label{dual_iso}
D(H,K) \otimes_{D(H',K)} V' \ra (\Ind^H_{H'}V)'_b \;, \hskip8pt \delta \otimes \vphi \mapsto \delta \cdot \vphi \;,
\end{numequation}

with $(\delta \cdot \vphi)(f) = \delta(g \mapsto \vphi(f(g^{-1})))$.
\end{para}

\vskip8pt

\begin{prop} Suppose $H = C \cdot H'$ with a compact locally $L$-analytic subgroup $C \sub H$, such that $C \cap H'$ is topologically finitely generated, and suppose, moreover, that $V$ is finite-dimensional. Then the map \ref{dual_iso} is an isomorphism.
\end{prop}

\Pf Then the same arguments as in \cite{ST3}, before Lemma 6.1, show that (\ref{dual_iso}) is an isomorphism of topological vector spaces, if we give the left side the quotient topology of the projective tensor product topology on $D(H,K) \otimes_K V'$. 
(We remark that the projective and inductive tensor product topologies coincide for tensor products of Fr\'echet spaces, cf. \cite[17.6]{S1}.) \qed 

\setcounter{enumi}{0}

\begin{para} We fix as in the previous chapter a standard parabolic subgroup ${\bf P}$ with Levi decomposition ${\bf P=L_P \cdot U_P}$ where ${\bf T \subset L_P}$. Let $\uM = (M, \tau)$ be an object of $\cO^P$. By \ref{Lemma_gen_W} we may choose a finite-dimensional locally analytic representation 
$W \sub M$ of $P$ which generates $M$ as 
$U(\frg)$-module. We let $\frd$ be the kernel of the canonical map $U(\frg) \otimes_{U(\frp)} W \ra M$, so that we have the tautological exact sequence

\vspace{-0.3cm}
\begin{numequation}\label{basicsequence}
0 \lra \frd \lra U(\frg) \otimes_{U(\frp)} W \lra M \lra 0 \;.
\end{numequation}

Recall from \cite{OS2} that we there is a pairing

\vspace{-0.3cm}
\begin{numequation}\label{pair_K}
\langle \cdot \,, \cdot \rangle: \left(D(G) \otimes_{D(P)} W\right)  \otimes_K \Ind^G_P(W') \lra K \;,
\end{numequation}

which identifies the left hand side with the topological dual of the right hand side and vice versa. We define the closed subspace $\Ind^G_P(W')^\frd \sub \Ind^G_P(W')$ by

\vspace{-0.3cm}
\begin{numequation}\label{basicdfn}
\Ind^G_P(W')^\frd = \{ f \in \Ind^G_P(W') \midc \mbox{ for all } \delta \in \frd : \langle \delta, f \rangle_{C^{an}(G,K)} = 0_{C^{an}(G,K)} \} \;.
\end{numequation}

Here we have used the pairing

\vspace{-0.3cm}
\begin{numequation}\label{pair_C(G,K)}
\begin{array}{rccc}
\langle \cdot , \cdot \rangle_{C^{an}(G,K)}: & \left(D(G) \otimes_{D(P)} W \right) \otimes_K \Ind^G_P(W') & \lra & C^{an}(G,K) \\
&&&\\
& (\delta \otimes w) \otimes f & \mapsto & \Big[ g \mapsto \big(\delta \cdot_r (f(\cdot)(w))\big)(g)\Big]
\end{array}
\end{numequation}

Here we have $\big(\delta \cdot_r (f(\cdot)(w))\big)(g) = \delta(x \mapsto f(gx)(w))$. The restriction $\Ind^G_P(W')|_{G_0}$ to $G_0 = \bG_0(O_L)$ is $\Ind^{G_0}_{P_0}(W')$, and is strongly admissible, as follows from \ref{dual_iso}. As a closed subrepresentation of $\Ind^G_P(W')$, the representation $\Ind^G_P(W')^\frd$ is therefore strongly admissible too. 
\end{para}

\begin{para} In order to conveniently describe the dual space of $\Ind^G_P(W')^\frd$, we consider the subring $D(\frg,P)$ generated by $U(\frg)$ and $D(P)$ inside $D(G)$. It is easy to see that on $M \in \cO^P$ there is a unique $D(\frg,P)$-module structure with the following properties: 

\vskip8pt

(i) The action of $U(\frp)$, as a subring of $U(\frg)$, coincides with the action of $U(\frp)$ as a subring of $D(P)$. \vskip5pt

(ii) The Dirac distributions $\delta_g \in D(P)$ act like group elements $g \in P$. 

\vskip8pt

Any morphism $M_1 \ra M_2$ in $\cO^P$ is automatically a homomorphism of $D(\frg,P)$-modules. 
\end{para}

\begin{prop}\label{basiciso} (i) The canonical map

$$\iota: M  = \left(U(\frg) \otimes_{U(\frp)} W\right)/\frd  \lra  \left(D(G) \otimes_{D(P)} W \right)/ \, D(G)\frd$$

\vskip8pt

extends to an isomorphism of $D(G)$-modules 

$$D(G) \otimes_{D(\frg,P)} M \cong \left(D(G) \otimes_{D(P)} W \right)/ \, D(G)\frd
\;.$$

\vskip8pt

(ii) The canonical map

$$\iota: M  = \left(U(\frg) \otimes_{U(\frp)} W\right)/\frd  \lra  \left(D(G_0) \otimes_{D(P_0)} W \right)/ \, D(G_0)\frd$$

\vskip8pt

extends to an isomorphism of $D(G_0)$-modules 

$$D(G_0) \otimes_{D(\frg,P_0)} M \cong \left(D(G_0) \otimes_{D(P_0)} W \right)/ \, D(G_0)\frd
\;.$$

\vskip8pt

(iii) There are canonical isomorphisms

$$D(G) \otimes_{D(\frg,P)} M  \cong D(G_0) \otimes_{D(\frg,P_0)} M \cong \left( \Ind^G_P(W')^\frd \right)' \,\,.$$

\vskip8pt

\end{prop}

\Pf The proofs of \cite[3.3, 3.7]{OS2} do not use that $M$ is in $\cO^\frp_\alg$ and carry over to the present setting. \qed

\vskip8pt

Denote by $\Rep^{\rm loc. an.}_K(G)$ the category of locally analytic representations of $G$ on 
barreled locally convex Hausdorff $K$-vector spaces. We define the functor

$$\cF^G_P: \cO^P \lra \Rep^{\rm loc. an.}_K(G) \hskip8pt \mbox{ by } \hskip8pt \cF^G_P(\uM) = \left( D(G) \otimes_{D(\frg,P)} M \right)' \;.$$

\vskip8pt

\begin{exam} Let $\tilde{\lambda}$ be a locally analytic character of $B$, and denote $\uM(\tilde{\lambda})$ the object of $\cO^B$ introduced in \ref{locanVerma}. Then we have $\cF^G_B(\uM(\tilde{\lambda})) = \Ind^G_B(\tilde{\lambda}^{-1})$. 
\end{exam}

\vskip8pt

\setcounter{enumi}{0}

\begin{prop}\label{exact} The functor $\cF^G_P$ is exact.
\end{prop}

\Pf  The proof of the analogous statement \cite[4.2]{OS2} for objects $M \in \cO^\frp_\alg$ applies in our current setting too (with some minor modifications). 

\vskip8pt

Alternatively we present here a slightly different proof. By the Schneider-Teitelbaum anti-equivalence of categories \cite[6.3]{ST2}, it suffices to show that the functor

$$\uM \mapsto D(G_0) \otimes_{D(\frg,P_0)} M$$

\vskip8pt

from $\cO^P$ into the category of $D(G_0)$-modules is exact. It is obviously right-exact. To see that it is left exact, let $\uM_1 = (M_1,\tau_1) \sub \uM$ be a subobject in $\cO^P$. Choose a finite-dimensional $P$-representation $W_1 \sub M_1$ which generates $M_1$ as $U(\frg)$-module, and choose a finite-dimensional $P$-representation $W \sub M$ which generates $M$ as $U(\frg)$-module and contains
$W_1$. Putting $M_2 = M/M_1$ and $W_2 = W/W_1$ we get a commutative diagram with exact rows and columns:

$$\xymatrixcolsep{2pc}\xymatrix{
         & 0 \ar[d] & 0 \ar[d] & 0 \ar[d] \\
0 \ar[r] & \frd_1 \ar[r] \ar[d] & U(\frg) \otimes_{U(\frp)} W_1 \ar[r] \ar[d] &
M_1 \ar[r] \ar[d] & 0 \\
0 \ar[r] & \frd \ar[r] \ar[d] & U(\frg) \otimes_{U(\frp)} W \ar[r] \ar[d] &
M \ar[r] \ar[d] & 0 \\
0 \ar[r] & \frd_2 \ar[r] \ar[d] & U(\frg) \otimes_{U(\frp)} W_2 \ar[r] \ar[d] &
M_2 \ar[r] \ar[d] & 0\\
         & 0  & 0  & 0 }$$

\vskip8pt

Here, by definition, $\frd_2 = \frd/\frd_1$, and the middle column of this diagram is exact, because each of the generalized Verma modules $U(\frg) \otimes_{U(\frp)} W_?$ is isomorphic to $U(\fru^-_\frp)_K \otimes_K W_? \,$, as a $K$-vector space. This gives rise to the following
diagram of $D(G_0)$-modules with exact rows:

$$\xymatrixcolsep{2pc}\xymatrix{
          & 0 \ar[d] & 0 \ar[d] \\
D(G_0) \otimes_{D(\frg,P_0)} \frd_1 \ar[r] \ar[d] &  D(G_0) \otimes_{D(P_0)} W_1 \ar[r] \ar[d] &
D(G_0) \otimes_{D(\frg,P_0)} M_1 \ar[r] \ar[d] & 0 \\
D(G_0) \otimes_{D(\frg,P_0)} \frd \ar[r] \ar[d] &  D(G_0) \otimes_{D(P_0)} W \ar[r] \ar[d] & D(G_0) \otimes_{D(\frg,P_0)} M \ar[r] \ar[d] & 0 \\
D(G_0) \otimes_{D(\frg,P_0)} \frd_2 \ar[r] \ar[d] &  D(G_0) \otimes_{D(P_0)} W_2 \ar[r] \ar[d] & D(G_0) \otimes_{D(\frg,P_0)} M_2 \ar[r] \ar[d] & 0 \\
        0 & 0 & 0 }$$

\vskip8pt

Here we have used that

$$D(G_0) \otimes_{D(\frg,P_0)} \left(U(\frg) \otimes_{U(\frp)} W_? \right)  = D(G_0) \otimes_{D(P_0)} W_? \,,$$

\vskip8pt

cf. \ref{basiciso}. The column in the middle is exact too. This follows from the Bruhat decomposition (cf. below). By the snake lemma, the right column is exact if the map

$$D(G_0) \otimes_{D(\frg,P_0)} \frd_2 \lra  D(G_0) \otimes_{D(P_0)} W_2$$

\vskip8pt

is injective.  In other words, we may assume that $M = U(\frg) \otimes_{U(\frp)} W$ is a generalized Verma module, where $W$ is equipped with a locally analytic $P$-representation which lifts the $\frp$-representation on $W$. In this way $M$ becomes an object of $\cO^P$. Let $M_1 \sub M$ be a subobject in $\cO^P$, and write $M_1$ as a quotient

$$0 \ra \frd_1 \ra U(\frg) \otimes_{U(\frp)} W_1 \ra M_1 \ra 0 \;,$$

\vskip8pt

so that we have an exact sequence

$$0 \ra \frd_1 \ra U(\frg) \otimes_{U(\frp)} W_1 \ra U(\frg) \otimes_{U(\frp)} W \;.$$

\vskip8pt

Our aim is to show that the following complex is exact:

$$D(G_0) \otimes_{D(\frg,P_0)} \frd_1  \ra  D(G_0) \otimes_{D(P_0)} W_1 \ra D(G_0) \otimes_{D(P_0)} W \,.$$

\vskip8pt

Note that in general we cannot choose $W_1$ to be a subspace of $W$ (in which case we are done).
Let $I \sub G_0$ be the Iwahori subgroup whose image in $\bG(O_L/(\pi_L))$ is $\bB(O_L/(\pi_L))$. Then we have the Bruhat decomposition

\vspace{-0.3cm}
\begin{numequation}\label{Bruhat}
G_0 = \coprod_{w \in W/W_P} IwP_0
\end{numequation}

and for every $w \in W$ an Iwahori decomposition

\vspace{-0.3cm}
\begin{numequation}\label{Iwahori}
I = U_w^- \cdot P_{0,w}  =  P_{0,w} \cdot U_w^-
\end{numequation}

with $U_w^- = I \cap wU^-w^{-1}$ and $P_{0,w} = I \cap wP_0w^{-1}$ (cf. \cite{Ti}, 3.1.1, \cite{OS2}, 3.3.2). Then we have

$$IwP_0 = U_w^-P_{w,0}wP_0 = U_w^-w(w^{-1}P_{w,0}w)P_0 = U_w^-wP_0 = w(w^{-1}U_w^-w)P_0 \,.$$

\vskip8pt

Put $U^{-,w} = w^{-1}U_w^-w$ which is a subgroup of $U^-$. It follows that

$$D(G_0) \otimes_{D(P_0)} W_1 = \bigoplus_{w \in W/W_P} \delta_w D(U^{-,w}) \otimes_K W_1$$

\vskip8pt

and similarly

$$D(G_0) \otimes_{D(P_0)} W = \bigoplus_{w \in W/W_P} \delta_w D(U^{-,w}) \otimes_K W \,.$$

\vskip8pt

These decompositions are compatible with the map $D(G_0) \otimes_{D(P_0)} W_1 \ra D(G_0) \otimes_{D(P_0)} W$. Hence it suffices to consider the map

$$D(U^{-,w}) \otimes_K W_1 \lra D(U^{-,w}) \otimes_K W \;,$$

\vskip8pt

which itself is the limit of the corresponding maps

\vspace{-0.3cm}
\begin{numequation}\label{with_r}
D_r(U^{-,w}) \otimes_K W_1 \lra D_r(U^{-,w}) \otimes_K W
\end{numequation}

as $r \in \sqrt{|K^*|}$ tends to 1 from below. The projective limit functor being exact in this case, cf. \cite{EGA}, 13.2.4, we are thus reduced to analyze the kernel of this latter map for $r$ sufficiently close to 1. Choose an open normal uniform pro-$p$ subgroup $H \sub U^{-,w}$, and define the norms on $D(H)$ by means of the associated canonical $p$-valuation, cf. \cite[2.2.3, 2.2.6]{OS1}. Define the norm $\|\cdot\|_r$ on $D(U^{-,w})$ by means of the decomposition $D(U^{-,w}) = \bigoplus_{u \in U^{-,w}/H} \delta_u D(H)$. Then we have $D_r(U^{-,w}) = \bigoplus_{u \in U^{-,w}/H} \delta_u D_r(H)$, and the map \ref{with_r} is a direct sum of maps

\vspace{-0.3cm}
\begin{numequation}\label{reduce_to_uniform}
D_r(H) \otimes_K W_1 \lra D_r(H) \otimes_K W \;.
\end{numequation}

Now let $\fru^{-,w} = \Lie(U^{-,w})$ and denote by $U_r(\fru^{-,w})$ the closure of $U(\fru^{-,w})$ in $D_r(H)$. Then $D_r(H)$ is a free $U_r(\fru^{-,w})$-module of finite length. More precisely, if $H_r = H \cap U_r(\fru^{-,w})$, then $D_r(H) = \bigoplus_{h \in H/H_r} \delta_h U_r(\fru^{-,w})$. The map \ref{reduce_to_uniform} is accordingly a direct sum of maps

\vspace{-0.3cm}
\begin{numequation}\label{reduce to enveloping}
U_r(\fru^{-,w}) \otimes_K W_1 \lra U_r(\fru^{-,w}) \otimes_K W \,.
\end{numequation}

Both sides are finite free $U_r(\fru^{-,w})$-modules and have as such a canonical topology induced by $U_r(\fru^{-,w})$. The given map is automatically strict. Moreover, the right hand side of \ref{reduce to enveloping} is simply the completion of $U(\fru^-) \otimes_K W$ with respect to a product norm $\|\cdot\|_r \otimes \|\cdot\|_W$, where $\|\cdot\|_W$ is some $K$-vector space norm on $W$, and the corresponding statement holds for the left hand side. Consider the exact sequence

$$0 \lra \frd_1 \cap \left(U(\fru^{-,w}) \otimes_K W_1\right) \lra U(\fru^{-,w}) \otimes_K W_1 \lra U(\fru^{-,w}) \otimes_K W \,.$$

\vskip8pt

By \cite[Cor. 6 in 1.1.9]{BGR} the corresponding sequence of completed modules is exact too. This shows that the kernel of \ref{reduce to enveloping} is equal to the closure of $\frd_1 \cap \left(U(\fru^{-,w}) \otimes_K W_1\right)$. This closure, however, is contained in the image of $D(G_0) \otimes_{D(\frg,P_0)} \frd_1$ which is a closed submodule of $D(G_0) \otimes_{D(P_0)} W_1$. \qed

\vskip8pt

As in \cite[Cor. 4.3]{OS2} one deduces that 

\vspace{-0.3cm}
\begin{numequation}\label{nonzero}
\mbox{if} \hskip8pt M \neq 0, \hskip8pt \mbox{then} \hskip8pt \cF^G_P(M) \neq 0 \;.
\end{numequation}

\setcounter{enumi}{0}

\begin{para} {\it Extending the functor $\cF^G_P$.} More generally, let $V$ be a $K$-vector space, equipped with a smooth admissible representation of the Levi subgroup $L_P \sub P$, and regard it via inflation as a representation of $P$. We recall that we always consider on smooth representations $V$ the finest locally convex $K$-vector space topology, i.e., the locally convex inductive limit of its finite-dimensional subspaces. 
As such $V$ is of compact type and furnishes a locally analytic $P$-representation.

\vskip8pt

Let $\uM = (M,\tau)$ be an object of $\cO^P$ and write $M$ as a quotient of a generalized Verma module

$$0 \ra \frd \ra U(\frg) \otimes_{U(\frp)} W \ra M \ra 0 \;,$$

\vskip8pt

with a finite-dimensional locally analytic $P$-representation $W \sub M$ which generates $M$ as a $U(\frg)$-modules, as in \ref{basicsequence}. As $W$ is a finite-dimensional space, the injective tensor product $W'\otimes_{K,\iota} V$ coincides with the projective tensor product $W'\otimes_{K,\pi} V$, and hence we simply write $W'\otimes_K V$ for it. It is complete and we have $W'\otimes_K V = \varinjlim_H W'\otimes_K V^H$, as locally convex vector spaces, where $H$ runs through all compact open subgroups of $P$. Equipped with the diagonal action of $P$, $W' \otimes_K V$ is a locally analytic representation. We set

\vspace{-0.3cm}
\begin{numequation}\label{explicit}
\begin{array}{rcl} \cF^G_P(\uM,W,V) & = & \Ind^G_P(W' \otimes_K V)^\frd \\
&&\\
& = & \{f \in \Ind^G_P(W' \otimes_K V) \midc \forall  \frz \in \frd: \langle \frz, f \rangle_{C^{an}(G,V)} = 0 \} \;,
\end{array}
\end{numequation}

where the pairing $\langle \cdot , \cdot \rangle_{C^{an}(G,V)}$ is defined as in \ref{pair_C(G,K)}. We are going to show that $\cF^G_P(\uM,W,V)$ is independent of the chosen $P$-representation $W$, in the sense of \ref{well-dfd} below. Later on we will therefore simplify notation by writing $\cF^G_P(\uM,V)$ instead of $\cF^G_P(\uM,W,V)$.
\end{para}

\vskip8pt

\setcounter{enumi}{0}

\begin{prop}\label{well-dfd} Let $G$, $P$, $M$, and $V$ be as above. Let $W_1 \sub W_2$ be two locally analytic 
finite-dimensional $P$-submodules of $M$ which generate $M$ as a $U(\frg)$-module. Then the canonical map
$\Ind^G_P(W_2' \otimes_K V) \ra \Ind^G_P(W_1' \otimes_K V)$ of  $G$-representations, induced by the map $W_2' \ra W_1'$, restricts to a topological isomorphism 

$$\cF^G_P(\uM,W_2,V) \stackrel{\cong}{\lra} \cF^G_P(\uM,W_1,V) \;.$$

\vskip8pt
\end{prop}

\Pf  For $i = 1,2$, let $0 \ra \frd_i \ra U(\frg)\otimes_{U(\frp)} W_i \ra M \ra 0$ be as in \ref{basicsequence}. 
It is immediate that under the (injective) map $U(\frg)\otimes_{U(\frp)} W_1 \ra U(\frg)\otimes_{U(\frp)} W_2$ induced by the inclusion $W_1 \sub W_2$ the submodule $\frd_1$ is mapped to $\frd_2$. It follows that the canonical map $\Ind^G_P(W_2' \otimes_K V) \ra \Ind^G_P(W_1' \otimes_K V)$ maps $\cF^G_P(\uM,W_2,V)$ into $\cF^G_P(\uM,W_1,V)$.

\vskip8pt

We fix a locally $L$-analytic section $s: G/P \ra G$ of the projection $G \ra G/P$ and let $\cH = s(G/P) \sub G$ be its image, so that we have an isomorphism of locally $L$-analytic manifolds $G \cong \cH \times P$ (this isomorphism depends on $s$). Using this isomorphism, we can identify the underlying topological vector space of $\Ind^G_P(W_i' \otimes_K V)$ with

$$C^{an}_L(\cH) \hat{\otimes}_K (W_i' \otimes_K V) = C^{an}_L(\cH) \otimes_K W_i' \otimes_K V \;,$$

\vskip8pt

 the completed tensor product on the left being the projective tensor product, cf. \cite[formula (56) and Remark 5.4]{K2}. This completed tensor product is equal to the ordinary tensor product, as $W' \otimes_K V$ is an inductive limit of finite-dimensional vector spaces. Using this isomorphism, which is natural in $W_i'$ and $V$, we
identify $f \in \Ind^G_P(W_i' \otimes_K V)$ with $\psi_f = \sum_k \psi_k \otimes \phi_k \otimes v_k \in C^{an}_L(\cH) \otimes_K W_i' \otimes_K V$, where $\psi_k \in C^{an}_L(\cH)$, $\phi_k \in W_i'$, $v_k \in V$. Then, for $f$ to be annihilated by $\frd_i$ translates into a condition on $\psi_f$ which is only a condition 
on $\sum_k \psi_k \otimes \phi_k$. With these identifications we find an isomorphism of $K$-vector spaces, natural in $W_i'$ and $V$,

$$\Ind^G_P(W_i' \otimes_K V)^{\frd_i} \stackrel{\cong}{\lra} \Big(C^{an}_L(\cH) \otimes_K W_i'\Big)^{\frd_i} \otimes_K V \;.$$

\vskip8pt

We get therefore a commutative diagram

$$\begin{array}{ccccc}
\Ind^G_P(W_2' \otimes_K V)^{\frd_2} & \stackrel{\cong}{\lra} & \Big(C^{an}_L(\cH) \otimes_K W_2'\Big)^{\frd_2} \otimes_K V & \stackrel{\cong}{\longleftarrow} & \Ind^G_P(W_2')^{\frd_2} \otimes_K V\\
\downarrow & & & & \downarrow\\
\Ind^G_P(W_1' \otimes_K V)^{\frd_1} & \stackrel{\cong}{\lra} & \Big(C^{an}_L(\cH) \otimes_K W_1'\Big)^{\frd_1} \otimes_K V  & \stackrel{\cong}{\longleftarrow} & \Ind^G_P(W_1')^{\frd_1} \otimes_K V\\
\end{array}$$

\vskip8pt

The vertical map on the right hand side is an isomorphism because the map $\Ind^G_P(W_2')^{\frd_2} \ra 
\Ind^G_P(W_1')^{\frd_1}$ is an isomorphism of topological vector spaces, by \ref{basiciso}. The map on the left must hence be an isomorphism too. \qed

\vskip8pt

Let $\uM \in \cO^P$ and $V$ be as above. Given two finite-dimensional locally analytic $P$-submodules $W_1, W_2 \sub M$, which both generate $M$ as a $U(\frg)$-module, we let $W_3 \sub M$ be another locally analytic finite-dimensional $P$-submodule containing both $W_1$ and $W_2$ (e.g., $W_3 = W_1 + W_2$). By \ref{well-dfd} the canonical maps

\vspace{-0.3cm}
\begin{numequation}\label{canonical-isos}
\cF^G_P(\uM,W_2,V) \longleftarrow \cF^G_P(\uM,W_3,V) \lra \cF^G_P(\uM,W_1,V)
\end{numequation}

are both isomorphisms of locally analytic $G$-representations. We have therefore a canonical isomorphism $\cF^G_P(\uM,W_2,V) \lra \cF^G_P(\uM,W_1,V)$ which does not depend on the choice of $W_3$. This is the unique isomorphism which is obtained from choosing any $W_3$ as above and inverting the map on the left of the resulting diagram \ref{canonical-isos}. Via these uniquely specified isomorphisms we can henceforth identify all representations $\cF^G_P(\uM,W,V)$ and write $\cF^G_P(\uM,V)$ for any one of them.

\vskip8pt

\begin{prop}\label{functor} Sending $(\uM,V)$ to $\cF^G_P(\uM,V)$ defines a bi-functor 

$$\cO^P \times \Rep^{\infty,{\rm adm}}_K(L_P) \longrightarrow  \Rep_K^{\rm loc. an.}(G) \;,$$

\vskip8pt

which is contravariant in $M$ and covariant in $V$. Here $\Rep^{\infty,{\rm adm}}_K(L_P)$ is the category of smooth admissible representations of the Levi subgroup $L_P \sub P$  on $K$-vector spaces.
\end{prop}

 \Pf Let $\alpha: \uM \ra \uN$ be a morphism in $\cO^P$, and let $\beta: U \ra V$ be a morphism in 
$\Rep^{\infty,{\rm adm}}_K(L_P)$. Choose a finite-dimensional locally analytic $P$-submodule $W_M \sub M$ which 
generates $M$ as a $U(\frg)$-module. Then choose a finite-dimensional locally analytic $P$-submodule $W_N \sub N$ 
which generates $N$ as a $U(\frg)$-module and which contains $\alpha(W_M)$. With the notation as in \ref{basicsequence} we get a commutative diagram

$$\begin{array}{ccccccccc}
0 & \ra & \frd_M & \ra & U(\frg)\otimes_{U(\frp)} W_M & \ra & M & \ra & 0\\
& & & & \downarrow & & \downarrow & & \\
0 & \ra & \frd_N & \ra & U(\frg)\otimes_{U(\frp)} W_N & \ra & N & \ra & 0
\end{array} $$

\vskip8pt

It follows that the map in the middle maps $\frd_M$ into $\frd_N$. This shows that the map 

$$\Ind^G_P(W_N' \otimes_K U) \lra \Ind^G_P(W_M' \otimes_K V)$$

\vskip8pt

induced by $\alpha' \otimes \beta:  W_N' \otimes_K U \lra W_M' \otimes_K V$ maps $\cF^G_P(\uN,U)$ into $\cF^G_P(\uM,V)$.
\qed

\vskip8pt

\begin{prop}\label{general admissibility}
(i) For all $\uM \in \cO^P$, and for all smooth admissible $L_P$-representations $V$ the $G$-representation $\cF^G_P(\uM,V)$ is admissible.\vskip5pt

(ii) If $V$ is of finite length, then $\cF^G_P(\uM,V)$ is strongly admissible for all $\uM \in \cO^P$. 
\end{prop}

\Pf This can be proved as in \cite[4.8]{OS2}. \qed

\vskip8pt

\begin{prop}\label{exact_in_both} a) The bi-functor $\cF^G_P$ is exact in both arguments.

\vskip8pt

b) If $Q \supset P$ is a parabolic subgroup, $\frq = \Lie(Q)$, and $\uM$ an object of $\cO^Q$, then

$$\cF^G_P(\uM,V) = \cF^G_Q(\uM,i^{L_Q}_{L_P(L_Q \cap U_P)}(V)) \;,$$

\vskip8pt

where $i^{L_Q}_{L_P(L_Q \cap U_P)}(V)=i^Q_P(V)$  denotes the corresponding induced representation in the category of smooth representations.
\end{prop}

\Pf The proof of \cite[4.9]{OS2} applies here too. \qed

\section{Irreducibility results}\label{irredresults}

In this section we prove analogues of the irreducibility results of \cite{OS2}, following the proofs given there very closely. In order to keep the exposition as self-contained as possible, we repeat many of the arguments. The key adjustment to the present setting (where we do not assume that the weights of $\frt$ on $M$ are integral) occurs in step four of the proof of theorem \ref{U_r_modules}. It is there where we use a result about integrality properties of highest weight modules, cf. section \ref{HWM}. 

\vskip8pt

The results in this section are valid under the following assumption on the residue characteristic $p$ of $L$: \vskip8pt

\begin{assumption}\label{hyp} If the root system $\Phi = \Phi(\frg,\frt)$ has irreducible components of type $B$, $C$ or $F_4$, we assume $p > 2$, and if $\Phi$ has irreducible components of type $G_2$, we assume that $p > 3$.
\end{assumption}

In section \ref{HWM} we prove certain results on highest weight modules under the same assumption, and these results are used in this section, which is why this restriction is imposed here. It might be possible that a more refined analysis in section \ref{HWM} shows that one can obtain results which are actually sufficient for our purposes here, without assuming \ref{hyp}. \vskip8pt

\begin{dfn}\label{maximal} Let $M$ be an object of the category $\cO$. We call a standard parabolic subalgebra $\frp$ {\it maximal for $M$} if $M \in \cO^\frp$ and if $M \notin \cO^\frq$ for any parabolic subalgebra $\frq$ strictly containing $\frp$.
\end{dfn}

It follows from \cite[sec. 9.4]{H2} that for every object $M$ of $\cO$ there is unique standard parabolic subalgebra $\frp$ which is maximal for $M$.
In this section, in order to clearly distinguish elements of the category $\cO^P$ from elements of $\cO$, we will write $\uM = (M,\tau)$ for objects of $\cO^P$.
 
\vskip8pt

\begin{thm}\label{irredG_0} Suppose \ref{hyp} holds. Let $\uM = (M,\tau: P \ra \End_K(M)^*)$ be an object of $\cO^P$ with the following two properties \vskip8pt

(a) $M$ is a simple object of $\cO$;

\vskip5pt

(b) $\frp$ is maximal for $M$. 

\vskip8pt

Then

\vskip8pt

(i) $D(G_0) \otimes_{D(\frg,P_0)}M$ is simple as a $D(G_0)$-module.

\vskip5pt

(ii) $\cF^G_P(\uM) = \Ind^G_P(W')^\frd$ is topologically irreducible as a $G_0$-representation.

\vskip5pt

(iii) $\cF^G_P(\uM) = \Ind^G_P(W')^\frd$ is topologically irreducible as a $G$-representation.
\end{thm}

 \Pf By \ref{basiciso} (iv) we know that $D(G_0) \otimes_{D(\frg,P_0)} M$ is a coadmissible $D(G_0)$-module. If it is simple, then, by \cite[6.3]{ST2}, the representation $\cF^G_P(\uM)$ is topologically irreducible as a $G_0$-representation. This shows that (i) implies (ii). Assertion (ii) obviously implies (iii). Finally, assertion (i) is a special case of Theorem \ref{irredH} below if we take $H = G_0$. \qed

\vskip8pt

\begin{para}\label{generalizations} {\it Generalization to open normal subgroups.} It will be useful later (in the proof of \ref{irredgeneral}), to have a similar criterion for the irreducibility of subrepresentations of $\cF^G_P(M)|_H$, 
where $H$ is an arbitrary open {\it normal} subgroup of $G_0$. For each $g \in G_0,$ the subspace

$$\left\{f \in \Ind^{G_0}_{P_0}(W')^\frd \midc \supp(f) \sub H g P_0\right\}$$

\vskip8pt

 of $\Ind_{P_0}^{G_0}(W')^\frd$ is closed and stable under the action of $H$, and we therefore have a decomposition of $H$-representations

$$\Ind_{P_0}^{G_0}(W')^\frd|_H \; = \; \bigoplus_{g \in H\bksl G_0/P_0}
\left\{f \in \Ind^{G_0}_{P_0}(W')^\frd \midc \supp(f) \sub H g P_0\right\} \;.$$

\vskip8pt

We define the representation $\Ind^{HP_0}_{P_0}(W')^\frd$ as in \ref{basicdfn}. Extending functions on $HP_0$ by zero gives an isomorphism of $H$-representations

$$\Ind^{HP_0}_{P_0}(W')^\frd \cong \left\{f \in \Ind^{G_0}_{P_0}(W')^\frd \midc \supp(f) \sub H P_0\right\} \;.$$

\vskip8pt

 For $g \in G_0$, we denote by $\Ind^{HP_0}_{P_0}(W')^{\frd,g}$ the space $\Ind^{HP_0}_{P_0}(W')^\frd$ equipped with the $H$-action defined by

$$(h\cdot_g f)(x) = f(g^{-1}h^{-1}g x) \;.$$


 The map

$$\begin{array}{ccc} \Ind^{HP_0}_{P_0}(W')^{\frd,g} & \lra & \left\{f \in \Ind^{G_0}_{P_0}(W')^\frd \midc \supp(f) \sub H g P_0\right\} \;, \\
& &\\
f & \longmapsto & x \mapsto \left\{\begin{array}{ccl} f(g^{-1}x) & , & x \in H g P_0\\
0 & , & \mbox{else}   \end{array} \right. \end{array}$$

\vskip8pt

 is then an isomorphism of $H$-representations. In the theorem below we use the following notation. 
For a $D(H)$-module $N$ and $g \in G_0,$ we denote by $\delta_g \star N$ the space $N$, equipped with the structure of a $D(H)$-module given by $\delta \cdot_g n = (\delta_{g^{-1}} \delta \delta_g)n$, where $n \in N$, $\delta \in D(H)$, and the product $\delta_{g^{-1}} \delta \delta_g$, computed as an element of $D(G_0)$, is in fact an element of $D(H)$.
\end{para}

\setcounter{enumi}{0}

\begin{thm}\label{irredH} Suppose \ref{hyp} holds. Let $\uM  = (M,\tau) \in \cO^P$ be as in \ref{irredG_0}. Let $H$ be an open normal subgroup of $G_0$, and let $g, g_1, g_2 \in G_0$. Then

\vskip8pt

(i) The dual space of $\Ind^{HP_0}_{P_0}(W')^{\frd,g}$ is isomorphic to $\delta_g \star \left(D(HP_0) \otimes_{D(\frg,P_0)} M\right)$ as $D(H)$-module. In particular, this is a coadmissible $D(H)$-module.

\vskip5pt

(ii) The $D(H)$-module $\delta_g \star \left(D(HP_0) \otimes_{D(\frg,P_0)} M\right)$ is simple.

\vskip5pt

(iii) The $D(H)$-modules $\delta_{g_1} \star \left(D(HP_0) \otimes_{D(\frg,P_0)} M\right)$ and $\delta_{g_2} \star \left(D(HP_0) \otimes_{D(\frg,P_0)} M\right)$ are isomorphic if and only if $g_1 HP_0 = g_2 HP_0$.

\vskip5pt

(iv) The $H$-representation $\Ind^{HP_0}_{P_0}(W')^{\frd,g}$ is topologically irreducible.

\vskip5pt

(v) The topological  $H$-representations $\Ind^{HP_0}_{P_0}(W')^{\frd, g_1}$ and $\Ind^{HP_0}_{P_0}(W')^{\frd, g_2}$ are isomorphic if and only if $g_1 HP_0 = g_2 HP_0$.
\end{thm}

\Pf (i) The proof of \ref{basiciso} only makes use of the fact that $G_0$ is open (hence its Lie algebra is equal to $\frg$) and that it contains $P_0$, and the corresponding statements carry over for any group with these properties. This shows statement (i) for the case when $g = 1$. From this the general case follows immediately.

\vskip8pt

(iv) This follows from (ii), by \cite[6.3]{ST2}.

\vskip8pt

(v) This follows from (iii), by \cite[6.3]{ST2}.

\vskip8pt

We will now start with the proofs of (ii) and (iii). The first step is to reduce to the case of a suitable subgroup $H_0 \sub H$ which is normal in $G_0$ and uniform pro-$p$. Having this accomplished, we rename $H_0$ to $H$. The second step is to pass to modules over the Banach algebra $D_r(H)$, cf. \cite[sec. 4]{ST2}. The third step consists in studying the restriction of these modules to the subring $U_r(\frg) = U_r(\frg,H)$, which is the completion of $U(\frg)$ with respect to the norm $\|\cdot\|_r$ on $D_r(H)$.

\vskip8pt

{\it Step 1: reduction to $H_0$.} The standard parabolic subgroup $\bP \sub \bG$ has a smooth model $\bP_0 \sub \bG_0$, and the unipotent radical $\bU_\bP^-$ of the parabolic subgroup opposite to $\bP$ has a smooth model $\bU_{\bP,0} \sub \bG_0$ as well.

\vskip8pt

Let $\Lie(\bG_0)$, $\Lie(\bP_0)$ and $\Lie(\bU_{\bP,0})$ be the Lie algebras of these (smooth) group schemes. These are $O_L$-lattices in $\frg = \Lie(G)$, $\frp = \Lie(P)$ and $\fru_P^- = \Lie(U_P^-)$, respectively. Moreover, $\Lie(\bG_0)$, $\Lie(\bP_0)$ and $\Lie(\bU_{\bP,0})$ are $\Zp$-Lie algebras, and we have

$$\Lie(\bG_0) = \Lie(\bU^-_{\bP,0}) \oplus \Lie(\bP_0) \;.$$

\vskip8pt

For $m_0 \ge 1$ ($m_0 \ge 2$ if $p=2$) the $O_L$-lattices $p^{m_0} \Lie(\bG_0)$, $p^{m_0} \Lie(\bP_0)$ and $p^{m_0} \Lie(\bU_{\bP,0})$ are powerful $\Zp$-Lie algebras, cf. \cite[sec. 9.4]{DDMS}, and hence
$\exp_G: \frg \dashrightarrow G$ converges on these $O_L$-lattices. Therefore,

\vspace{-0.3cm}
\begin{numequation}\label{H_0}
H_0 = \exp_G\Big(p^{m_0} \Lie(\bG_0)\Big) \;, \;\; H_0^+ = \exp_G\Big(p^{m_0} \Lie(\bP_0)\Big) \;, \;\; H_0^- = \exp_G\Big(p^{m_0} \Lie(\bU^-_{\bP,0})\Big)
\end{numequation}

are uniform pro-$p$ groups. The adjoint action of $G_0 = \bG_0(O_L)$ leaves $\Lie(\bG_0)$ invariant, and hence $p^{m_0} \Lie(\bG_0)$. This shows that $H_0$ is normal in $G_0$, and analogous arguments show that $H_0^+$ and $H_0^-$ are normal in $P_0$ and $U_{P,0}^-$, respectively. It follows from the existence of 'coordinates of the second kind' that $H_0 = H_0^-H_0^+$ and $H_0 \cap P_0 = H_0^+$ and $H_0 \cap U^-_{P,0} = H_0^-$, cf. \cite[2.2.4 (ii)]{OS1}.

\vskip8pt

Moreover, for large enough $m_0$ the group $H_0$ is contained in $H$, and 

$$D(H) = \bigoplus_{h \in H/H_0} \delta_h D(H_0) \;.$$

\vskip8pt

This implies that, as $D(H_0)$-modules,

\vspace{-0.3cm}
\begin{numequation}\label{reduction1}
\delta_g \star \left(D(HP_0) \otimes_{D(\frg,P_0)} M\right) = \bigoplus_{h \in HP_0/H_0P_0} \delta_{gh} \star \left(D(H_0P_0) \otimes_{D(\frg,P_0)} M\right) \;.
\end{numequation}

Let us now assume that statements (ii) and (iii) hold for $H_0$. Then the direct sum on the right of \ref{reduction1} consists of simple $D(H_0)$-modules which are pairwise non-isomorphic. As they are permuted by the action of $H$, it follows that the left side is simple as $D(H)$-module.
Furthermore, if $\delta_{g_1} \star \left(D(HP_0) \otimes_{D(\frg,P_0)} M\right) \cong \delta_{g_2} \star \left(D(HP_0) \otimes_{D(\frg,P_0)} M\right)$ as $D(H)$-module, then also as $D(H_0)$-module. But then we must have $g_1 h_1 H_0P_0 = g_2 h_2 H_0 P_0$ for some $h_1, h_2 \in H$, in particular $g_1 H P_0 = g_2 H P_0$.

\vskip8pt

Thus we have shown that statements (ii) and (iii) for $D(H)$ follow from the corresponding statements for $D(H_0)$.

\vskip8pt

To simplify notation, we will from now on assume that $H$, and its subgroups $H^+$ and $H^-$, are defined as in \ref{H_0}, and are, in particular, uniform pro-$p$ groups.

\vskip8pt

{\it Step 2: passage to $D_r(H)$.} We put

$$\kappa = \left\{\begin{array}{lcl} 1 & , & p>2 \\
2 & , & p=2 \end{array} \right. \;.$$

\vskip8pt

(In \cite{OS1}, which we are going to use in the following, $\kappa$ is denoted by $\vep_p$.) Let $r$ always denote a real number in $(0,1) \cap p^\bbQ$ with the property:

\vspace{-0.3cm}
\begin{numequation}\label{r_and_s}
\hskip-20pt \mbox{there is $m \in \Z_{\ge 0}$ such that $s = r^{p^m}$ satisfies $\frac{1}{p} < s$ and $s^{\kappa} < p^{-1/(p-1)}$ .}
\end{numequation}

For the definition of the canonical $p$-valuation on a uniform pro-$p$ group we refer to \cite[2.2.3]{OS1}.
We let $\|\cdot\|_r$ denote the norm on $D(H)$ associated to $r$ and the canonical $p$-valuation, cf. \cite[2.2.6]{OS1} (where this norm is denoted by $\bar{q}_r$). $D_r(H)$ denotes the corresponding Banach space completion. This is a noetherian Banach algebra, and $D(H) = \varprojlim_{r <1} D_r(H)$ (cf. \cite{ST2} for more information). If $\tH \sub G_0$ is a compact open subgroup which contains $H$, then $H$ is normal in $\tH$

\vspace{-0.3cm}
\begin{numequation}\label{dcomp}
D(\tH) = \bigoplus_{g \in \tH/H} \delta_g D(H)
\end{numequation}

and we let $\|\cdot\|_r$ be the maximum norm on $D(\tH)$ given by

\vspace{-0.3cm}
\begin{numequation}\label{maxnorm}
\|\sum_{g \in \tH/H} \delta_g \lambda_g \|_r = \max\{\|\lambda_g \|_r \midc g \in \tH/H \} \;.
\end{numequation}

Then the completion $D_r(\tH)$ of $D(\tH)$ with respect to $\|\cdot\|_r$ has an analogous decomposition:

$$D_r(\tH) = \bigoplus_{g \in \tH/H} \delta_g D_r(H)$$

\vskip8pt

and therefore $D_r(\tH) =  D_r(H) \otimes_{D(H)} D(\tH)$. We are going to use the preceding discussion in the case when $\tH = HP_0$. Put $\bM = D(HP_0) \otimes_{D(\frg,P_0)} M$. To show that $\bM$ is a simple $D(H)$-module it suffices, by \cite[3.9]{ST2}, to show that

$$\bM_r := D_r(H) \otimes_{D(H)} \bM = D_r(HP_0) \otimes_{D(\frg,P_0)} M$$

\vskip8pt

is a simple $D_r(H)$-module for a sequence of $r$'s converging to $1$. By \ref{basiciso} and \ref{nonzero}

$$\cF^{G_0}_{P_0}(M)' = D(G_0) \otimes_{D(HP_0)} \bM$$

\vskip8pt

is non-zero, and thus $\bM = \varprojlim_{r<1} \bM_r$ is non-zero. Hence we must have $\bM_r \neq 0$ for $r$ sufficiently close to $1$. As the image of $M$ in $\bM_r$ generates $\bM_r$ as $D_r(HP_0)$-module, and because $M$ is simple, the canonical map $M \ra \bM_r$ must be injective when $\bM_r \neq 0$. From now on we assume that, in addition to \ref{r_and_s}, $r$ is such that $\bM_r \neq 0$, and we consider $M$ as being contained in $\bM_r$.

\vskip8pt

{\it Step 3: passage to $U_r(\frg)$.} Let $U_r(\frg) = U_r(\frg,H)$ be the topological closure of $U(\frg)$ in $D_r(H)$. It is important for our approach to have a useful description of $U_r(\frg)$. We will freely use the following fact, which follows from \cite[5.6]{Sch}:

\vspace{-0.3cm}
\begin{numequation}\label{density1}
\mbox{$U(\frg)$ is dense in $D_r(H)$ if $r^{\kappa} < p^{-\frac{1}{p-1}}$. In this case $U_r(\frg) = D_r(H)$.}
\end{numequation}

Let $(P_m(H))_{m \ge 1}$ be the lower $p$-series of $H$, cf. \cite[1.15]{DDMS}. Note that $P_1(H) = H$. For $m \ge 0$ put $H^m : = P_{m+1}(H)$ so that $H^0 = H$. We refer to \cite[sec. 4.5, sec. 9.4]{DDMS} for the notion of a $\Zp$-Lie algebra of a uniform pro-$p$ group. It follows from \cite[3.6]{DDMS} that $H^m$ is a uniform pro-$p$ group with $\Zp$-Lie algebra equal to $p^m\Lie_\Zp(H)$. Let $s = r^{p^m}$ be as in \ref{r_and_s}. Denote by $\|\cdot\|_s^{(m)}$ the norm on $D(H^m)$ associated to $s$ and the canonical $p$-valuation on $H^m$. Then, by \cite[6.2, 6.4]{Sch}, the restriction of $\|\cdot\|_r$ on $D(H)$ to $D(H^m)$ is equivalent to $\|\cdot\|_s^{(m)}$, and $D_r(H)$ is a finite and free $D_s(H^m)$-module a basis of which is given by any set of coset representatives for $H/H^m$. By \ref{density1} we can conclude:

\vspace{-0.3cm}
\begin{numequation}\label{density2}
\mbox{if $s = r^{p^m}$ is as in \ref{r_and_s}, then $U(\frg)$ is dense in $D_s(H^m)$, hence $U_r(\frg,H) = D_s(H^m)$.}
\end{numequation}

In particular, $U_r(\frg) \cap H = H^m$ is an open subgroup of $H$. Let

$$\frm_r := U_r(\frg)M$$

\vskip8pt

be the $U_r(\frg)$-submodule of $\bM_r$ generated by $M$. Because we assume $\bM_r \neq 0$ we also have $\frm_r \neq 0$, and $M$ is contained in $\frm_r$.

\vskip8pt

Likewise, we equip $D(H^+)$ ($D(H^-)$, resp.) with the norm $\|\cdot\|_r$ associated to the canonical $p$-valuation on $H^+$ ($H^-$, resp.), and give $D(P_0)$ ($D(U^-_{P,0})$, resp.) the maximum norm as above (cf. \ref{dcomp} and \ref{maxnorm}). Again, we denote by $D_r(P_0)$ ($D_r(U^-_{P,0})$, resp.) the corresponding completions. As is easy to see, these norms on $D(P_0)$ ($D(U^-_{P,0})$, resp.) are equal to the restriction of the norms on $D(HP_0)$ ($D(U^-_{P,0}H)$, resp.) coming from the norm $\|\cdot\|_r$ on $D(H)$ by the recipe explained above.

\vskip8pt

Let $U_r(\frp)$ be the closure of $U(\frp)$ in $D_r(H^+)$. Then $U_r(\frp)$ is the completion of $U(\frp)$ with respect to the norm associated to the canonical $p$-valuation on $H^+$. Because the canonical $p$-valuation of an element of $h \in H$ ($h \in H^+$, resp.) can be read off from $\log_G(h) \in p^{m_0}\Lie(\bG_0)$ ($\log_G(h) \in p^{m_0}\Lie(\bP_0)$, resp.), the canonical $p$-valuation on $H^+$ is the restriction of the canonical $p$-valuation on $H$. Hence $U_r(\frp)$ is also the closure of $U(\frp)$ in $U_r(\frg)$.

\vskip8pt

It follows from \ref{density2} (applied to $H^+$ and $\frp = \Lie_\Zp(H^+)$) that $D_r(P_0)$ is generated as a module over $U_r(\frp)$ by finitely many Dirac distributions $\delta_{g_1}, \ldots, \delta_{g_k}$, with $g_i \in P_0$. Because $P_0$ acts on $M$, the $U_r(\frg)$-module $\frm_r$ is actually a module over the subring

\vspace{-0.3cm}
\begin{numequation}\label{D_r_frg_P_0}
D_r(\frg,P_0) = U_r(\frg)D_r(P_0) = \sum_{i=1}^k U_r(\frg) \cdot \delta_{g_i}
\end{numequation}

generated by $U_r(\frg)$ and $D_r(P_0)$ inside $D_r(G_0)$. Moreover, $\frm_r$ is a finitely generated $U_r(\frg)$-module (in fact, generated by a single vector of highest weight), hence carries a canonical $U_r(\frg)$-module topology, as $U_r(\frg)$ is a noetherian Banach algebra (cf. \cite[1.4.2]{K1} which applies here because we assume $r \in (\frac{1}{p},1) \cap p^\bbQ$).

\vskip8pt

The module $M$ is clearly dense in $\frm_r$ with respect to this topology. It follows from \cite[1.3.12]{F}
or \cite[3.4.8]{OS1} that $\frm_r$ is a simple $U_r(\frg)$-module, and in particular a simple $D_r(\frg,P_0)$-module. We recall the following result \cite[5.6]{OS2}:

\setcounter{enumi}{0}

\begin{sublemma}\label{P_{0,r}} Let $r$ and $s$ be as in \ref{r_and_s}. Put $P_{0,r} = HP_0 \cap D_r(\frg,P_0)$. Then

\vskip8pt

(i) $P_{0,r} = H^mP_0 = H^{-,m}P_0$, where $H^{-,m} = P_m(H^-)$.

\vskip5pt

(ii) $D_r(HP_0) =  \bigoplus_{g \in HP_0/P_{0,r}} \delta_g D_r(\frg,P_0)$. \qed
\end{sublemma}

\vskip8pt

From the second assertion we deduce that $\frm_r$ is equal to $D_r(\frg,P_0) \otimes_{D(\frg,P_0)} M$ as a $D_r(\frg,P_0)$-submodule of $D_r(HP_0) \otimes_{D(\frg,P_0)} M = \bM_r$. Using \ref{P_{0,r}} (ii) again we we can conclude

\vspace{-0.3cm}
\begin{numequation}\label{decomp}
\bM_r = D_r(HP_0) \otimes_{D(\frg,P_0)} M = D_r(HP_0) \otimes_{D_r(\frg,P_0)} \frm_r = \bigoplus_{g \in HP_0/P_{0,r}} \delta_g \frm_r \,.
\end{numequation}

Note that, as $U_r(\frg)$-modules, the submodule $\delta_g \frm_r$ on the right hand side is the same as $\delta_g \star \frm_r$, where we use the notation as introduced before \ref{irredH} (for $U_r(\frg)$ instead of $D(H)$). By Theorem \ref{U_r_modules} below the modules $\delta_g \star \frm_r$ are all simple, and there is no isomorphism of $U_r(\frg)$-modules $\delta_{g_1} \star \frm_r \cong \delta_{g_2} \star \frm_r$ if $g_1P_{0,r} \neq g_2P_{0,r}$. This shows that $\bM_r$ is a simple $D_r(H)$-module.

\vskip8pt

Now assume that $\delta_{g_1} \star \bM$ and $\delta_{g_2} \star \bM$ are isomorphic as $D(H)$-modules. Then, after tensoring with $D_r(H)$ we find that $\delta_{g_1} \star \bM_r$ and $\delta_{g_2} \star \bM_r$ are isomorphic as $D_r(H)$-modules, hence as $U_r(\frg)$-modules. Then \ref{decomp}, together with \ref{U_r_modules}, implies that $g_1h_1 P_{0,r} = g_2 h_2 P_{0,r}$ for some $h_1, h_2 \in H$, and hence $g_1HP_0 = g_2HP_0$.

\vskip8pt

We have therefore reduced statements (ii) and (iii) of \ref{irredH} to the assertions of the theorem below. \qed

\setcounter{enumi}{0}

\begin{thm}\label{U_r_modules} Assume \ref{hyp} holds. Let $\uM = (M,\tau) \in \cO^P$ be as in \ref{irredG_0}. Let $H = \exp_G\left(p^{m_0}\Lie(\bG_0)\right)$ be the uniform pro-$p$ introduced in \ref{H_0} (where it was called $H_0$). Assume that $r \in (\frac{1}{p},1) \cap p^{\bbQ}$ and $s$ are as in \ref{r_and_s}. Assume that

\vspace{-0.3cm}
$$\frm_r = D_r(\frg,P_0) \otimes_{D(\frg,P_0)} M \neq 0 \;.$$


 (This will be the case if $r$ is close enough to 1.) Then, for every $g \in G_0$ the $U_r(\frg)$-module $\delta_g \star \frm_r$ is simple. For any $g_1, g_2 \in G_0$ with $g_1P_{0,r} \neq g_2P_{0,r}$ the $U_r(\frg)$-modules $\delta_{g_1} \star \frm_r$ and $\delta_{g_2} \star \frm_r$ are not isomorphic.
\end{thm}

 \Pf We continue to use the notation introduced in the proof of \ref{irredH}. We have already seen that $\frm_r$ is simple as $U_r(\frg)$-module, and this implies that $\delta_g \star \frm_r$ is simple as well. It is trivial to check that for $x, y \in G_0$ one has $\delta_x \star (\delta_y \star \frm_r) = \delta_{xy} \star \frm_r$, and if $x \in P_{0,r}$, then $\delta_x \star \frm_r \ra \frm_r$, $n \mapsto \delta_x \cdot n$, is an isomorphism of $U_r(\frg)$-modules.

\vskip8pt

Now let $\phi: \delta_{g_1} \star \frm_r \stackrel{\cong}{\lra} \delta_{g_2} \star \frm_r$ be an isomorphism of $U_r(\frg)$-modules. A straightforward calculation shows that $\phi$ is also a $U_r(\frg)$-module isomorphism $\delta_{g_2^{-1}g_1} \star \frm_r \stackrel{\cong}{\lra} \frm_r$, so that we may assume $g_2 = 1$. Let $I_0 \sub G_0$ be the Iwahori subgroup whose image in $\bG(O_L/(\pi_L))$ is $\bB(O_L/(\pi_L))$.\footnote{We denote the Iwahori subgroup by $I_0$ because we use $I$ for the subset of simple roots corresponding to $\frp$.} Using the Bruhat decomposition

\vspace{-0.3cm}
$$G_0 = \coprod_{w \in W/W_P} I_0wP_0$$


 we may write $g = g_1= h^{-1}wp_1$ with $h \in I_0$, $w \in W$ and $p_1 \in P_0$.
By the Iwahori decomposition

\vspace{-0.3cm}
$$I_0 = (I_0 \cap {U_{P,0}^-}) \cdot (I_0 \cap P_0)$$


 we can write $h = up_2$ with $p_2 \in I_0 \cap P_0$ and $u \in I_0 \cap U_{P,0}^-$. The same reasoning as before then shows that an isomorphism $\delta_g \star \frm_r \stackrel{\cong}{\lra} \frm_r$ induces an isomorphism of $U_r(\frg)$-modules which we again denote by $\phi$:

$$\phi: \delta_w \star \frm_r \stackrel{\cong}{\lra} \delta_u \star \frm_r \,.$$

\vskip8pt

{\it Step 1.} We show first that this can only happen if $w \in W_P$.  Let $\lambda \in \frt^*$ be the highest weight of $M$.

\vskip8pt

Then by Cor. \ref{locfinitesubalgebra} and the maximality condition with respect to $\frp$, the parabolic subalgebra $\frp$ is $\frp_I$, where the root system of the Levi subalgebra of $\frp_I$ has $I$ as a basis of simple roots. Suppose $w$ is not contained in $W_I = W_P$. Then there is a positive root  $\beta \in \Phi^+\setminus\Phi^+_I$ such that $w^{-1}\beta < 0$, hence $w^{-1}(-\beta) > 0$, cf. \cite[2.3]{Car}. Consider a non-zero element element $y \in \frg_{-\beta}$, and let $v^+ \in M$ be a weight vector of weight $\lambda$ (uniquely determined up to a non-zero scalar). Then we have the following identity in $\delta_w \star \frm_r$,

\vspace{-0.3cm}
$$y \cdot_w v^+ = \Ad(w^{-1})(y) \cdot v^+ = 0$$


 as $\Ad(w^{-1})(y) \in \frg_{-w^{-1}\beta}$ annihilates $v^+$. We have $\phi(v^+) = v$ for some nonzero $v \in \frm_r$. And therefore

\vspace{-0.3cm}
$$0 = \phi(y \cdot_w v^+) = y \cdot_u \phi(v^+) = y \cdot_u v = \Ad(u^{-1})(y) \cdot v \;.$$


 On the other hand, $y' := \Ad(u^{-1})(y)$ is an element of $\fru_P^-$ and as such acts injectively on $M$, cf. Cor. \ref{injective}. We show that $y'$ actually acts injectively on $\frm_r$, too. Indeed, let $v\in \frm_r$ with $y'\cdot v=0.$ Write $v$ as a convergent sum $v = \sum_{\mu \in \Lambda(v)} v_{\lambda-\mu}$ where $\Lambda(v)$ is a (non-empty) subset of $\Z_{\ge 0} \alpha_1 \oplus \ldots \oplus \Z_{\ge 0} \alpha_\ell$ and $v_{\lambda-\mu} \in M_{\lambda-\mu} \setminus \{0\}$ is a vector of weight $\lambda - \mu$ (cf. the fourth assertion in \cite[1.3.12]{F}). Here $\Delta = \{\alpha_1, \ldots, \alpha_\ell\}$ is the set of simple roots. Write $y' = \sum_{\gamma \in B} y_\gamma$, where $B$ is a non-empty subset of $\Phi^+\setminus \Phi^+_I$ and $y_\gamma$ is a non-zero element of $\frg_{-\gamma} \subset \fru_P^-$. Then we have

$$0 = y' \cdot v = \sum_{\nu \in \Z_{\ge 0} \alpha_1 \oplus \ldots \oplus \Z_{\ge 0} \alpha_\ell} \left(\sum_{\mu \in \Lambda(v), \gamma \in B, \nu = \mu+\gamma} y_\gamma \cdot v_{\lambda-\mu}\right) \;,$$

\vskip8pt

where

$$\sum_{\mu \in \Lambda(v), \gamma \in B, \nu = \mu+\gamma} y_\gamma \cdot v_{\lambda-\mu}$$

\vskip8pt

lies in $M_{\lambda-\nu}$. Because of the uniqueness of the expansions of elements of $\frm_r$ as convergent series of weight vectors, cf. \cite[1.3.12]{F}, we can conclude that $\sum_{\mu \in \Lambda(v), \gamma \in B, \nu = \mu+\gamma} y_\gamma \cdot v_{\lambda-\mu}$ vanishes.

\vskip8pt

Define on $\Z_{\ge 0} \alpha_1 \oplus \ldots \oplus \Z_{\ge 0} \alpha_\ell$ the following lexicographic ordering:

$$\sum_{i=1}^\ell n_i \alpha_i > \sum_{i=1}^\ell n_i' \alpha_i \hskip10pt \iff \hskip10pt \exists k \ge 1: n_i = n_i' \mbox{ for } 1 \le i \le k-1 \mbox{ and } n_k > n_k' \;.$$

\vskip8pt

Choose $\gamma^+ \in B$ and $\mu^+ \in \Lambda(v)$, both minimal with respect to this ordering. With $\nu^+ = \gamma^+ + \mu^+$ we then have

$$0 = \sum_{\mu \in \Lambda(v), \gamma \in B, \nu^+ = \mu+\gamma} y_\gamma \cdot v_{\lambda-\mu} = y_{\gamma^+} \cdot v_{\lambda-\mu^+} \;.$$

\vskip8pt

This contradicts Cor. \ref{injective}, and we can thus conclude that $w$ must be an element of $W_P$.

\medskip

From now on we may assume that $w=1$ since $W_P \subset P_0 \sub P_{0,r}$.

\vskip8pt

{\it Step 2.} Now let $u \in U^-_{P,0}$ and let $\phi: \frm_r \stackrel{\cong}{\lra} \delta_u \star \frm_r$ be an isomorphism of $U_r(\frg)$-modules. We suppose that $u$ is not contained in $P_{0,r}$ and are seeking to derive a contradiction. Let $v^+ \in M_\lambda \setminus \{0\}$ be a vector of highest weight as above. Put $\phi(v^+) = v$ with some non-zero $v \in \frm_r$. Then we have for any $x \in \frt$, the following identity in $\delta_u \star \frm_r$:

$$\lambda(x)v = \phi(x \cdot v^+) = x \cdot_u \phi(v^+) = \Ad(u^{-1})(x) \cdot v \;.$$

\vskip8pt

 We have thus for all $x \in \frt$, the following identity in $\frm_r$

\vspace{-0.3cm}
\begin{numequation}\label{identity}
\lambda(x)v = \Ad(u^{-1})(x) \cdot v \hskip5pt .
\end{numequation}

Note that this equation only involves the action of elements of $\fru_P^- \oplus \frt$, because $\Ad(u^{-1})(x)$ is in $\fru_P^- \oplus \frt$. Next we embed $\frm_r$ into the ''formal completion'' of $M$, i.e.,

$$\hat{M} = \prod_{\mu} M_\mu$$

\vskip8pt

by mapping the weight spaces $M_\mu \sub \frm_r$ to the corresponding space $M_\mu \sub \hat{M}$ (cf. the fourth assertion in \cite[1.3.12]{F}).
Then $\hat{M}$ is in an obvious way a module for $U(\fru_P^- \oplus \frt)$. This module structure extends to a representation of $U^-_P $ as follows. Since $U_P^-$ is a unipotent group, the exponential $\exp_{U_P^-}: \fru_P^- \dashrightarrow U_P^-$ is actually defined on the whole Lie algebra, so that we have a bijective map $\exp_{U_P^-}: \fru_P^- \ra U_P^-$. Let $\log_{U_P^-}: U_P^- \ra \fru_P^-$ be the inverse map. Then we can define for $u \in U_P^-$ and $v = \sum_\mu v_\mu \in \hat{M}$:

$$\delta_u \cdot v = \sum_\mu \sum_{n \ge 0} \frac{1}{n!} \log_{U_P^-}(u)^n \cdot v_\mu \,.$$


 Note that this sum is well-defined in $\hat{M}$, because $\log_{U_P^-}(u)$ is in $\fru_P^-$, and there are hence only finitely many terms of given weight in this sum. (We continue to write the action of an element $u \in U_P^-$ on $\hat{M}$ by $\delta_u$.) Moreover, it gives an action of $U_P^-$ on $\hat{M}$ because

$$\begin{array}{rcl}
\exp(\log_{U_P^-}(u_1)) \cdot \exp(\log_{U_P^-}(u_2)) &
= & \exp\big(\cH(\log_{U_P^-}(u_1),\log_{U_P^-}(u_2))\big) \\
\\
& = & \exp(\log_{U_P^-}(u_1u_2)) \;,
\end{array}$$

\vskip5pt

 where $\cH(X,Y) = \log(\exp(X)\exp(Y))$ is the Baker-Campbell-Hausdorff series (which converges on $\fru_P^-$, as $\fru_P^-$ is nilpotent). It is then immediate that this action is compatible with the action of $U(\fru_P^- \oplus \frt)$. The identity \ref{identity} implies then the following identity in $\hat{M}$

\vspace{-0.3cm}
$$\delta_{u^{-1}} \cdot( x \cdot (\delta_u \cdot v)) = \Ad(u^{-1})(x) \cdot v = \lambda(x)v \hskip5pt ,$$


 for all $x \in \frt$, and thus, multiplying both sides of the previous equation with $\delta_u$,

\vspace{-0.3cm}
$$x \cdot (\delta_u \cdot v) = \lambda(x) \delta_u \cdot v \hskip5pt .$$

 Hence $\delta_u \cdot v \in \hat{M}$ is a weight vector of weight $\lambda$ and must therefore be equal to a non-zero scalar multiple of $v^+$. After scaling $v$ appropriately we have $\delta_u \cdot v = v^+$ or
$v = \delta_{u^{-1}} \cdot v^+$ with

\vspace{-0.3cm}\begin{numequation}\label{series}
\delta_{u^{-1}} \cdot v^+ = \sum_{n \ge 0} \frac{1}{n!} (-\log_{U_P^-}(u))^n \cdot v^+ =: \Sigma \hskip5pt .
\end{numequation}

 Our goal is to show that the series $\Sigma$, which is an element of $\hat{M}$, does in fact not lie in the image of $\frm_r$ in $\hat{M}$, if $u \notin P_{0,r}$, thus achieving a contradiction.

\vskip8pt

{\it Step 3.} For the rest of this section we fix a Chevalley basis $(x_\beta,y_\beta,h_\alpha \midc \beta \in \Phi^+, \alpha \in \Delta)$ of $\frg' = [\frg,\frg]$, cf. \cite[Thm. in sec. 25.2]{H2}. We have $x_\beta \in \frg_\beta$, $y_\beta \in \frg_{-\beta}$, and $h_\alpha = [x_\alpha,y_\alpha] \in \frt$, for $\alpha \in \Delta$. We then let $\frg'_\Z$ be the $\Z$-span of these basis elements. $\frg'_\Z$ is a Lie algebra over $\Z$. Then

$$\Lie_\Zp(H^-) = p^{m_0}\Lie(\bU^-_{\bP,0}) = \bigoplus_{\beta \in \Phi^+ \setminus \Phi^+_I} O_L y_\beta^{(0)} \;,$$


 where $y_\beta^{(0)} = p^{m_0}y_\beta$. Moreover, the $\Zp$-Lie algebra of $H^{-,m}$ is

\vspace{-0.3cm}
$$\Lie_\Zp(H^{-,m}) = p^m\Lie(H^-) = \bigoplus_{\beta \in \Phi^+ \setminus \Phi^+_I} O_L y_\beta^{(m)} \;,$$


 where $y_\beta^{(m)} = p^{m_0+m}y_\beta$.
We assume that $r$ and $s$ are as in \ref{r_and_s}. Then \ref{density2} shows that $U_r(\fru_P^-) = U_r(\fru_P^-,H^-) = D_s(H^{-,m})$. Elements in $U_r(\fru_P^-)$ thus have a description as power series in $(y^{(m)}_\beta)_{\beta \in \Phi^+ \setminus \Phi^+_I}$:

$$U_r(\fru_P^-) = \left\{\sum_{n = (n_\beta)} d_n  (y^{(m)})^n \midc \lim_{|n| \ra \infty} |d_n|s^{\kappa|n|} = 0 \right\} \;,$$

\vskip8pt

 where $(y^{(m)})^n$ is the product of the $(y^{(m)}_\beta)^{n_\beta}$ over all $\beta \in \Phi^+ \setminus \Phi^+_I$, taken in some fixed order.
Let $\|\cdot\|_s^{(m)}$ be the norm on $D_s(H^{-,m})$ induced by the canonical $p$-valuation on $H^{-,m}$. Then we have for any generator $y^{(m)}_\beta$

\vspace{-0.3cm}
\begin{numequation}\label{norm_generator}
\| y^{(m)}_\beta\|_s^{(m)} = s^\kappa \;,
\end{numequation}

 as follows from \cite[5.3, 5.8]{Sch}. We recall that by \ref{P_{0,r}} we have $P_{0,r} = H^{-,m}P_0$. Write

$$\log_{U_P^-}(u) = \sum_{\beta \in B(u)} z_\beta \;,$$

 with a non-empty set $B(u) \sub \Phi^+ \setminus \Phi^+_I$ and non-zero elements $z_\beta \in \frg_{-\beta}$. Put

$$B^+(u) = \{\beta \in B(u) \midc z_\beta \notin O_L y^{(m)}_\beta \} \;.$$

 This is a non-empty subset of $B(u)$ since $u \notin P_{0,r}$. Put $B'(u) = B(u) \setminus B^+(u)$,

$$z^+ = \sum_{\beta \in B^+(u)} z_\beta \hskip10pt \mbox{ and } \hskip10pt z' = \sum_{\beta \in B'(u)} z_\beta = \log_{U_P^-}(u) - z^+ \;.$$

 Then $z' \in \Lie_\Zp(H^{-,m})$, and thus $\exp(z') \in H^{-,m}$. The element $u_1 = u\exp(z')^{-1} = u\exp(-z')$ does not lie in $H^{-,m}$, and $\delta_u \star \frm_r \cong \delta_{u_1} \star \frm_r$. We may hence replace $u$ by $u_1 = u \exp(-z')$. Then we compute $\log_{U^-}(u_1) = \log_{U^-}(u \exp(-z'))$ by means of the Baker-Campbell-Hausdorff formula. Because of the commutators appearing in this formula we have

\vspace{-0.3cm}\begin{numequation}\label{log}
\log_{U^-}(u_1) \in z^+ + \sum \left( \, \mbox{iterated commutators of } \frg_{-B^+(u)} \mbox{ with } \frg_{-B'(u)} \, \right) \hskip5pt ,
\end{numequation}

 where $\frg_{-B^+(u)} = \bigoplus_{\beta \in B^+(u)} \frg_{-\beta}$ and analogous for $\frg_{-B'(u)}$.
Recall that the {\it height} $ht(\beta)$ of the root $\beta = \sum_{\alpha \in \Delta}n_\alpha \alpha$ is defined to be the sum $\sum_{\alpha \in \Delta}n_{\alpha}$.
Put $ht'(u) = \min\{ht(\beta) \midc \beta \in B'(u)\}$.  It follows from the preceding formula \ref{log} that
if the right summand does not vanish we have  $ht'(u_1) > ht'(u)$.

\vskip8pt

The process of passing from $u$ to $u_1$ can be repeated finitely many times, but will finally produce an element $u_N \in U^-_{P,0} \setminus H^{-,m}$ which has the property that $B'(u_N) = \emptyset$ (and hence $u_{N+1} = u_N$). We will denote this element again by $u$.

\vskip8pt

Next we chose an extremal element $\beta^+$ among the $\beta \in B(u) = B^+(u)$, i.e. one of the minimal generators of the cone $\sum_{\beta \in B(u)} \bbR_{\ge 0}\beta$
inside $\bigoplus_{\alpha \in \Delta} \bbR \alpha$. Then we have

$$\bbR_{>0}\beta^+ \cap \sum_{\beta \in B(u), \beta \neq \beta^+} \bbR_{\ge 0}\beta \hskip5pt = \hskip5pt \emptyset \,.$$

\vskip8pt

This means in particular that no positive multiple of $\beta^+$ can be written as a linear combination $\sum_{\beta \in B(u), \beta \neq \beta^+} c_\beta \beta$ with non-negative integers $c_\beta \in \Z_{\ge 0}$. It follows that if $n$ is a positive integer and

\vspace{-0.3cm}
\begin{numequation}\label{implication}
\begin{array}{l}
n\beta^+ = \gamma_1 + \ldots + \gamma_m \hskip5pt \mbox{ with not necessarily distinct } \hskip5pt \gamma_i \in B(u) \\
\\
\Rightarrow \hskip8pt \left[m=n \hskip5pt \mbox{ and } \hskip5pt \gamma_1 = \ldots = \gamma_n = \beta^+ \right] \hskip3pt .
\end{array}
\end{numequation}

After these intermediate considerations we return to the series

$$\Sigma = \sum_{n \ge 0} \frac{1}{n!} (-\log_{U_P^-}(u))^n \cdot v^+ = \sum_{n \ge 0} \frac{(-1)^n}{n!} (z_{\beta^+} + z_{\beta_2} + \ldots + z_{\beta_k})^n \cdot v^+$$

\vskip8pt

where $B(u) = \{\beta^+, \beta_2, \ldots, \beta_k \}$. It follows from \ref{implication} and Cor. \ref{injective}  that if we write $\Sigma$ as a (formal) sum of weight vectors, there is for any $n \in \Z_{\ge 0}$ a unique non-zero weight vector in this series which is of weight $\lambda - n \beta^+$, and this element is $\frac{(-1)^n}{n!}z_{\beta^+}^n \cdot v^+$.

\vskip8pt

{\it Step 4.} The last part of the proof is to show that the formal sum of weight vectors

\vspace{-0.3cm}
\begin{numequation}\label{expseries}
\sum_{n \ge 0} \frac{(-1)^n}{n!}z_{\beta^+}^n \cdot v^+
\end{numequation}

cannot be the corresponding sum of weight components of an element of $\frm_r$, when considered as an element of $\hat{M}$ and written as a sum of weight vectors.

\vskip8pt

Write $\Phi^+ = \{\beta_1, \ldots, \beta_t\}$ such that $\Phi_I^+ = \{\beta_{\tau+1}, \ldots, \beta_t\}$. (Note: the roots $\beta_i$ are not necessarily the same as in Step 3.) Choose finitely many weight vectors $v_j \in M$ of weight $\lambda - \mu_j$, $j=1,\ldots,k$, generating $M$ as $U(\fru^-_P)$-module. We can and will assume that $v_j$ is of the form $\left(\prod_{\eta=\tau+1}^t (y_{\beta_\eta}^{(0)})^{\nu_{j,\eta}}\right) \cdot v^+$.

\vskip8pt

Put $y_i = y_{\beta_i}$ and $y^{(m)}_i = y^{(m)}_{\beta_i}$ for $i=1, \ldots, \tau$. To show that \ref{expseries} does not converge to an element in $\frm_r$, we write

\vspace{-0.3cm}
\begin{numequation}\label{basicrelation}
\frac{(y_{\beta^+}^{(0)})^n}{n!}.v^+ = \sum_{1 \le j \le k} \sum_{\nu \in \cI_{n,j}} c_{\nu,j}^{(n)} (y_1^{(0)})^{\nu_1} \cdot \ldots \cdot (y_\tau^{(0)})^{\nu_\tau}.v_j \;,
\end{numequation}

where $\cI_{n,j} \sub \Z_{\ge 0}^\tau$ consists of those $\nu = (\nu_1,\ldots,\nu_\tau)$ such that $\mu_j + \nu_1\beta_1 + \ldots + \nu_\tau\beta_\tau = n\beta^+$. From now on we assume that $m_0$ is such that

$$\left|p^{m_0}\lambda(h_\tau)\right| \le 1$$

\vskip8pt

for all $\tau \in \Delta$. This means that the assumptions of Prop. \ref{both_conditions} are in place. Hence, by the statement of this proposition, there is at least one index $(\nu,j) = (\nu_1, \ldots, \nu_\tau, j)$ with the property that

\vspace{-0.3cm}
\begin{numequation}\label{estimate_c}
|c_{\nu,j}^{(n)}|_K \ge \left|\frac{1}{n!}\right|_K  \hskip10pt \mbox{and} \hskip10pt \nu_1 + \ldots + \nu_\tau + \sum_{k=\tau+1}^t \nu_{j,k} \ge n \,.
\end{numequation}

We will use this inequality to estimate the $||\cdot||_r$-norm of any lift of $\frac{1}{n!}z_{\beta^+}^n.v^+$ to

$$\bigoplus_{1 \le j \le k} U_r(\fru_P^-) \otimes_K Kv_j\;.$$

\vskip8pt

Here the $\|\cdot\|_r$ - norm on this free $U_r(\fru_P^-)$-module is the supremum norm of the $||\cdot||_r$-norms on each 
summand, defined by

$$\|\delta \otimes v_j\|_r := \|\delta\|_r \;,$$

\vskip8pt

for $\delta \in U_r(\fru_P^-)$. Because $B'(u)=\emptyset$ (cf. Step 3), we have $z_{\beta^+} = p^{-a}y^{(m)}_{\beta^+}$ for some $a>0$. We then get from \ref{basicrelation}

$$\begin{array}{rcl} \frac{z_{\beta^+}^n}{n!}.v^+ & = & p^{-na} \frac{(y^{(m)}_{\beta^+})^n}{n!} \cdot v^+ \\
&&\\
& = & \sum_{1 \le j \le k} \sum_{\nu \in \cI_{n,j}} c_{\nu,j}^{(n)} p^{-na + m(n- \nu_1-\ldots-\nu_\tau)} (y^{(m)}_1)^{\nu_1} \cdot \ldots \cdot (y^{(m)}_\tau)^{\nu_\tau}.v_j \;.
\end{array}$$

\vskip8pt

It follows from \cite[6.2, 6.4]{Sch} that the restriction of the norm $\|\cdot\|_r$ on $D(H,K)$ to $D(H^m,K)$ is equivalent to the norm $\|\cdot\|^{(m)}_s$ on $D(H^m,K)$ (cf. also \cite[7.4]{Sch}, where it is stated that the induced topologies are equivalent). Therefore, we are now going to estimate the $\|\cdot\|^{(m)}_s$-norm of the term

$$c_{\nu,j}^{(n)} p^{-na + m(n- \nu_1-\ldots-\nu_r)} (y^{(m)}_1)^{\nu_1} \cdot \ldots \cdot (y^{(m)}_r)^{\nu_r} \;,$$

\vskip8pt

where $(\nu,j)$ is such that \ref{estimate_c} holds. By \ref{norm_generator}, the $\|\cdot\|^{(m)}_s$-norm of this term is greater or equal to

$$\begin{array}{cl}
& \left|\frac{1}{n!}\right|_K p^{na + m(\nu_1+\ldots+\nu_r-n)} s^{\nu_1+\ldots+\nu_r} \\
 & \\
= & \left|\frac{1}{n!}\right|_K p^{na- n + (m-1)(\nu_1+\ldots+\nu_r-n)} (ps)^{\nu_1+\ldots+\nu_r} \,.
\end{array}$$

\vskip8pt

Note that \ref{estimate_c} implies that $n - (\nu_1 + \ldots + \nu_r)$ is bounded from above by some constant $C_1$. Hence we get

$$p^{na- n + (m-1)(\nu_1+\ldots+\nu_r-n)} \ge p^{n(a-1)} \cdot p^{-(m-1)C_1}$$

\vskip8pt

is bounded away from $0$ (recall that $a \in \Z_{>0}$ and $m$ is fixed). And because $s> \frac{1}{p}$ we have $ps >1$, and the term $(ps)^{\nu_1+\ldots+\nu_r}$ is unbounded as $n \ra \infty$ ($\nu_1 + \ldots + \nu_\tau \ge n - C_1$). Altogether we get that any lift of $\frac{1}{n!}z_{\beta^+}^n.v^+$ to
$\bigoplus_{1 \le j \le t} U_r(\fru^-) \otimes_K Kv_i$ has an $\|\cdot\|_r$ - norm which exceeds $C_2(ps)^n$, where $C_2>0$ is some constant (cf. the relation between the norms $\|\cdot\|_r$ and $\|\cdot\|^{(m)}_s$ mentioned above).

\vskip8pt

The sum $\sum_{n \ge 0} \frac{1}{n!}z_{\beta^+}^n.v^+$, which is an element of $\hat{M}$ is therefore not contained in the image of $\frm_r$ (under the injection $\frm_r \hra \hat{M}$). And as we have seen before, this in turn proves that $\sum_{n \ge 0} \frac{(-\log(u))^n}{n!}.v^+ = \delta_{u^{-1}}.v^+ \in \hat{M}$ is not in the image of $\frm_r$. \qed

\vskip8pt

\begin{thm}\label{irredgeneral} Assume \ref{hyp} holds. Let $\uM = (m,\tau) \in \cO_P$ be such that $M$ is simple, and suppose that $\frp$ is maximal 
for $M$. Let $V$ be a smooth and irreducible $L_P$-representation. Then $\cF^G_P(M,V) = \Ind^G_P(W'\otimes_K V)^\frd$ is topologically irreducible as a $G$-representation.
\end{thm}

\Pf The proof of the analogous result \cite[5.8]{OS2} does not use that $M$ is an object of $\cO^\frp_\alg$, except through references to \cite[5.5]{OS2}, which we have now proved without the assumption that $M$ is an object of $\cO^\frp_\alg$, cf. \ref{irredH}. \qed

\section{Integrality properties of relations in simple modules in $\cO$}\label{HWM}

The goal of this section is to prove a result about relations in simple modules $M \in \cO$ which generalizes \cite[8.13]{OS2}. For the sake of a clear exposition of this technical section, we have found it convenient (if not necessary) to repeat many of the arguments given in \cite[sec. 8]{OS2}. We begin by quoting some results from loc.cit. which do not involve integrality statements. 

\vskip8pt

Let $A$ be any associative unital ring. For $x \in A$ we let $\ad(x): A \ra A$ be defined by $\ad(x)(z) = [x,z] = xz-zx$. In the following we also write $[x^{[i]},z]$ for $\ad(x)^i(z)$, the value on $z$ of the $i$-th iteration of $\ad(x)$.

\vskip8pt

\begin{lemma}\label{generallemma} Let $x,z_1, \ldots, z_n \in A$. For all $k \in \Z_{\ge 0}$ one has

$$x^k \cdot z_1z_2 \ldots z_n = \sum_{i_1 + \ldots + i_{n+1} = k} {k \choose i_1  \ldots  i_{n+1}} [x^{[i_1]},z_1]\cdot \ldots \cdot [x^{[i_n]},z_n]x^{i_{n+1}}$$

\vskip8pt

where, as usual,

$${k \choose i_1  \ldots  i_{n+1}} = \frac{k!}{(i_1!) \cdot \ldots \cdot (i_{n+1}!)}$$

\vskip8pt

\noindent Similarly,

$$[x^{[k]}, z_1z_2 \ldots z_n] = \sum_{
i_1 + \ldots + i_n = k} {k \choose i_1  \ldots  i_n} [x^{[i_1]},z_1]\cdot \ldots \cdot [x^{[i_n]},z_n] \;.$$ \qed
\end{lemma}

\vskip8pt

\begin{lemma}\label{lemma1} Let $x \in \frg$, let $M$ be a
$U(\frg)$-module and $v \in M$.

\vskip8pt

\noindent (i) If $x$ acts locally finitely on $v$ (i.e., the $K$-vector space generated by $(x^i.v)_{i \ge 0}$ is finite-dimensional), then $x$ acts locally finitely on $U(\frg).v$.

\vskip8pt

\noindent (ii) If $x.v = 0$ and $[x,[x,y]] = 0$ for some $y \in \frg$, then

$$x^ny^n.v = n! [x,y]^n.v \;.$$ \qed

\vskip8pt
\end{lemma}

\vskip8pt

As before, we let $\Phi = \Phi(\frg,\frt)$ be the root system and

$$\Phi^+ = \{\beta_1, \ldots, \beta_t\} \hskip10pt \mbox{and} \hskip10pt \Delta = \{\alpha_1, \ldots, \alpha_\ell\} \sub \Phi^+$$

\vskip8pt

be a set of positive roots and a corresponding basis.

\vskip8pt

\begin{lemma}\label{lemma2} Assume $\gamma \in \Phi^+$ is not simple and write  $\gamma = \alpha + \beta$ with $\alpha \in \Delta$ and $\beta \in \Phi^+$. Suppose that $i\beta - j \alpha$ is in $\Phi^+$ for some $i,j \in \Z_{>0}$.
\vskip8pt

\noindent (i) Then $(i\beta - j\alpha) - \gamma  =  (i-1)\beta -(j+1)\alpha$ is {\bf either} a positive root {\bf or} not in $\Phi \cup \{0\}$. Therefore: {\bf either} $[\frg_{i\beta - j\alpha}, \frg_{-\gamma}]$ is equal to a root space $\frg_\chi$ with $\chi = i'\beta - j'\alpha \in \Phi^+$ for some $i',j' \in \Z_{>0}$ {\bf or} $[\frg_{i\beta - j\alpha}, \frg_{-\gamma}] = 0$.
\vskip8pt

\noindent (ii) Moreover, $(i\beta - j\alpha) - \alpha  =  i\beta -(j+1)\alpha$ is {\bf either} a positive root {\bf or} not in $\Phi \cup \{0\}$. Therefore: {\bf either} $[\frg_{i\beta - j\alpha}, \frg_{-\alpha}]$ is equal to a root space $\frg_\chi$ with $\chi = i'\beta - j'\alpha \in \Phi^+$ for some $i',j' \in \Z_{>0}$ {\bf or} $[\frg_{i\beta - j\alpha}, \frg_{-\alpha}] = 0$.
\vskip8pt

\noindent (iii) Let $M$ be a $U(\frg)$-module and $v \in M$ be annihilated by the radical $\fru$ of $\frb$. Let $x \in \frg_\beta$ and $y \in \frg_{-\gamma}$. Then, for any sequence of non-negative integers $i_1, \ldots, i_n$ we have

\vspace{-0.3cm}
$$[x^{[i_1]},y]\cdot \ldots \cdot [x^{[i_n]},y].v = 0$$


\noindent if there is at least one $i_j > 1$. \qed
\end{lemma}

\vskip8pt

\begin{prop}\label{notlocnilp} Let $\frp = \frp_I$ for some $I \sub \Delta$. Suppose $M \in \cO^\frp$ is a highest weight module with highest weight $\lambda$ and

$$I = \{\alpha \in \Delta \midc \langle \lambda, \alpha^\vee \rangle \in \Z_{\ge 0} \} \;.$$

\vskip8pt

Then no non-zero element of $\fru_\frp^-$ acts locally finitely on $M$. \qed
\end{prop}

\vskip8pt

\begin{cor}\label{injective} Let $\frp = \frp_I$ for some $I \sub \Delta$. Suppose $M \in \cO^\frp$ is a simple module of highest weight $\lambda$ and

$$I = \{\alpha \in \Delta \midc \langle \lambda, \alpha^\vee \rangle \in \Z_{\ge 0} \} \;.$$

\vskip8pt

Then the action of any non-zero element of $\fru_\frp^-$ on $M$ is injective. \qed
\end{cor}

\vskip8pt

\begin{cor}\label{locfinitesubalgebra} Let $M \in \cO$ be a highest weight module. Then the set of elements in $\frg$ which act locally finitely on $M$ is a standard parabolic Lie subalgebra of $\frg$. If $M$ has highest weight $\lambda$, then this standard parabolic subalgebra is $\frp_I$ where

$$I = \{\alpha \in \Delta \midc \langle \lambda, \alpha^\vee \rangle \in \Z_{\ge 0} \} \;.$$

\vskip8pt

$\frp_I$ is maximal for $M$ in the sense of \ref{maximal}. \qed
\end{cor}

\vskip8pt

{\it Integrality properties.} For the rest of this section we fix a Chevalley basis $(x_\beta,y_\beta,h_\alpha \midc \beta \in \Phi^+, \alpha \in \Delta)$ of $\frg' = [\frg,\frg]$, cf. \cite[Thm. in sec. 25.2]{H2}. Here we have $x_\beta \in \frg_\beta$ and $y_\beta \in \frg_{-\beta}$. We then let $\frg'_\Z$ be the $\Z$-span of these basis elements. $\frg'_\Z$ is a Lie algebra over $\Z$. For later use we quote from \cite[sec. 25]{H2} the following facts.

\vskip8pt

\begin{prop}\label{Chevalley} Let $\beta$, $\beta'$ be linearly independent roots. Let $z_\beta$ be $x_\beta$, if $\beta$ is positive, and let $z_\beta$ be $y_{-\beta}$, if $\beta$ is negative. Define $z_{\beta'}$ and $z_{\beta+\beta'}$ in the same manner, if $\beta + \beta'$ is a root.
Let $\beta'-r\beta, \beta'-(r-1)\beta, \ldots, \beta'+q\beta$ ($r \ge 0$, $q \ge 0$) be the $\beta$-string through $\beta'$. Then: \vskip5pt

(i) $r-q = \langle \beta, (\beta')^\vee \rangle$.

\vskip5pt

(ii) $[z_\beta,z_{\beta'}] = \pm(r+1)z_{\beta+\beta'}$ if $\beta + \beta'$ is a root. \qed
\end{prop}

\begin{para}\label{intro} Let $M$ be a simple module in the category $\cO$, and assume that $\frp = \frp_I$ is maximal for $M$. Denote by $v^+$ a highest weight vector of weight $\lambda$. We will be studying relations

$$(y_\gamma^{(0)})^n \cdot v^+ = \sum_{\nu \in \cI_n} c_\nu (y_1^{(0)})^{\nu_1} \cdot \ldots \cdot (y_t^{(0)})^{\nu_t} \cdot v^+ \,,$$

\vskip8pt

where

\vskip8pt

-  $\gamma \in \Phi^+ \setminus \Phi_I^+$,

\vskip5pt

- $m_0$ is a non-negative integer with the property that $|p^{m_0}\lambda(h_\alpha)| \le 1$ for all $\alpha \in \Delta$, 

\vskip5pt

- $y_\gamma^{(0)} = p^{m_0}y_\gamma$, and $y_i^{(0)} = p^{m_0}y_{\beta_i}$ for $i = 1,\ldots,t$, 

\vskip5pt

- $\cI_n$ consists of all $t$-tuples $(\nu_1, \ldots, \nu_t) \in \Z_{\ge 0}^t$ satisfying $\nu_1\beta_1 + \ldots +\nu_t \beta_t = n\gamma$,

\vskip5pt

- $c_v$ are coefficients in $K$.

\vskip8pt

Our aim (cf. \ref{both_conditions}) is to show that there is at least one $\nu$ with the following two properties

$$\begin{array}{ll}
i. & \nu_1 + \ldots + \nu_t \ge n \,, \\
ii. & |c_\nu|_K \ge 1 \;.
\end{array}$$

\vskip8pt

\noindent We start by showing the existence of some $\nu \in \cI_n$ with the second property.
\end{para}

\setcounter{enumi}{0}

\begin{lemma}\label{estimate} Suppose the residue characteristic of $K$ does not divide any of the non-zero numbers among $\langle \beta, \alpha^\vee\rangle$, $\alpha, \beta \in \Phi$, $\alpha \neq \pm \beta$. Let $M
\in \cO^\frp$, $\frp = \frp_I$, $\gamma \in \Phi^+ \setminus \Phi^+_I$, and $m_0$ be as in \ref{intro}. Then, for any $n \in \Z_{\ge 0}$ and any expression

\begin{numequation}\label{relation}
(y_\gamma^{(0)})^n \cdot v^+ = \sum_{\nu \in \cI_n} c_\nu (y_1^{(0)})^{\nu_1} \cdot \ldots \cdot (y_t^{(0)})^{\nu_t} \cdot v^+ \,,
\end{numequation}
 
there is at least one index
$\nu \in \cI_n$ such that $|c_\nu|_K \ge 1$.
\end{lemma}

\Pf We start by recalling that for any $i \ge 0$ and $\beta \in \Phi^+$ the endomorphism of $\frg$:

$$\frac{1}{i!}[x_\beta^{[i]}, \cdot ] = \frac{1}{i!}\ad(x_\beta)^i$$

\vskip8pt

preserves the $\Z$-form $\frg'_\Z$ of $\frg'$, cf. \cite[Prop. in sec. 25.5]{H2}. Denote by $U(p^{m_0}\frg_\Z')$ the enveloping algebra of $p^{m_0}\frg'_\Z$, i.e., the quotient of the tensor algebra $T_\Z(p^{m_0}\frg'_\Z)$ by the two-sided ideal generated by all elements of the form $xy-yx-[x,y]$ with $x,y \in p^{m_0}\frg'_\Z$. It follows from \ref{generallemma} that $\frac{1}{i!}\ad(x_\beta)^i$ also preserves $U(p^{m_0}\frg'_\Z)$. 

\vskip8pt

The proof proceeds by induction on $ht(\gamma)$. Suppose $ht(\gamma) = 1$ and let $\beta_{i} = \gamma$. Then the set $\cI_n$ consists of a single element $\nu$ which is the $t$-tuple that has the entry $n$ in the $i^{\mbox{\tiny th}}$ place and zeros elsewhere. The right hand side of \ref{relation} is thus $c_\nu (y_\gamma^{(0)})^n.v^+$. By Cor. \ref{injective} the element $y_{\gamma}$ acts injectively on $M$, and we therefore get $c_\nu = 1$.

\vskip8pt

Now we assume that $ht(\gamma) > 1$. Write $\gamma = \alpha + \beta$ with a simple root $\alpha \in \Delta$ and a positive root $\beta$.
Clearly, not both $\alpha$ and $\beta$ can be contained in $\Phi_I$. We distinguish two cases.

\vskip8pt

(a) Suppose $\beta - \alpha$ is not in $\Phi$. As $\beta \neq \alpha$ (the root system $\Phi$ is reduced) we 
have $[\frg_\alpha,\frg_{-\beta}] = [\frg_{-\alpha},\frg_\beta] = \{0\}$. Then, if $\alpha \notin I$, we consider $x_\beta$, the element 
of the Chevalley basis which generates $\frg_\beta$. We have by Lemma \ref{lemma1}:

$$x_\beta^n(y_\gamma^{(0)})^n.v^+ = n![x_\beta,y_\gamma^{(0)}]^n.v^+ \;.$$

\vskip8pt

Consider the $\beta$-string through $-\gamma$:  $-\gamma-r\beta, \ldots, -\gamma + q\beta$. Then  $-\alpha-(r+1)\beta, \ldots, -\alpha+(q-1)\beta$ is the $\beta$-string through $-\alpha$, and because we assume here that $-\alpha+\beta \notin \Phi$, we deduce that $q=1$. By \ref{Chevalley} we then conclude $r+1 = \langle -\alpha, \beta^\vee \rangle$ and $[x_\beta,y_\gamma] = \pm \langle \alpha, \beta^\vee \rangle y_\alpha$. Hence

$$x_\beta^n(y_\gamma^{(0)})^n.v^+ = n!( \pm \langle \alpha, \beta^\vee \rangle)^n  (y_\alpha^{(0)})^n.v^+ \;.$$

\vskip8pt

As $\langle \alpha, \beta^\vee \rangle = -(r+1) \neq 0$ and because of our assumption on the residue characteristic of $K$, the integer $\langle \alpha, \beta^\vee \rangle$ is invertible in $O_K$. On the other hand, equation \ref{relation} gives

$$\begin{array}{rcl}
x_\beta^n.(y_\gamma^{(0)})^n.v^+ & = & \sum_{\nu \in \cI_n} c_\nu x_\beta^n \cdot (y_1^{(0)})^{\nu_1} \cdot \ldots \cdot (y_t^{(0)})^{\nu_t}).v^+ \\
&&\\
& = & \sum_{\nu \in \cI_n} c_\nu \ad(x_\beta)^n((y_1^{(0)})^{\nu_1} \cdot \ldots \cdot (y_t^{(0)})^{\nu_t}).v^+ \;.
\end{array}$$

\vskip8pt

We recall that $\ad(x_\beta)^n((y_1^{(0)})^{\nu_1} \cdot \ldots \cdot (y_t^{(0)})$ is in $n!U(p^{m_0}\frg'_\Z)$. By assumption, cf. \ref{intro}, we have $p^{m_0}h_\tau.v^+ = p^{m_0}\lambda(h_\tau) v^+ \in \Zp v^+$, for all $\tau \in \Delta$. Because the PBW theorem holds for $U(p^{m_0}\frg'_\Z)$, by \cite{Bourbaki_Lie_1_3}, we thus find that $x_\beta^n.(y_\gamma^{(0)})^n.v^+$
is of the form

$$n! \sum_{\nu' \in \cI'_n} c'_{\nu'} (y_1^{(0)})^{\nu_1'} \cdot \ldots \cdot (y_t^{(0)})^{\nu_t'}.v^+ \;,$$

\vskip8pt

where $\cI'_n$ consists of all $\nu' \in \Z_{\ge 0}^t$ such that $\nu_1'\beta_1 + \ldots + \nu_t'\beta_t = n\alpha = n(\gamma - \beta)$, and
the numbers $c'_{\nu'}$ are linear combinations of the $c_\nu$ with coefficients in $\Zp$. We therefore get

$$(y_\alpha^{(0)})^n.v^+ = \frac{1}{(\pm \langle \alpha, \beta^\vee \rangle)^n} \sum_{\nu' \in \cI'_n} c'_{\nu'} (y_1^{(0)})^{\nu_1'} \cdot \ldots \cdot (y_t^{(0)})^{\nu_t'}.v^+ \; .$$

\vskip8pt

The induction hypothesis shows that at least one of the coefficients $c'_{\nu'}$ must be of absolute value at least $1$, and this implies that at least one of the coefficients $c_\nu$ is of absolute value at least $1$.

\vskip8pt

Now suppose $\alpha \in I$. Then $\beta \notin \Phi_I$ and we consider $x_\alpha$, the member of the Chevalley basis which generates $\frg_\alpha$. By Lemma \ref{lemma1}:

$$x_\alpha^n (y_\gamma^{(0)})^n.v^+ = n![x_\alpha,y_\gamma^{(0)}]^n.v^+ \;.$$

\vskip8pt

The same arguments as above give $[x_\alpha,y_\gamma] = \pm \langle \beta, \alpha^\vee \rangle y_\beta$, and $\langle \beta, \alpha^\vee \rangle$ is a unit in $O_K$. As before, we then multiply the right hand side of \ref{relation} with $x_\alpha^n$ and find that

$$(y_\beta^{(0)})^n.v^+ = \frac{1}{(\pm \langle \beta, \alpha^\vee \rangle)^n} \sum_{\nu' \in \cI'_n} c'_{\nu'} (y_1^{(0)})^{\nu_1'} \cdot \ldots \cdot (y_t^{(0)})^{\nu_t'}.v^+ \; ,$$

\vskip8pt

where the numbers $c'_{\nu'}$ are linear combinations of the $c_\nu$ with coefficients in $\Zp$. And we conclude again by induction.

\vskip8pt

(b) Suppose $\beta - \alpha$ is in $\Phi$. Then it must be in $\Phi^+$, and we have $\gamma - k\alpha \in \Phi^+$ for $0 \le k \le k_0$ (with $k_0 \le 3$, cf. \cite[0.2]{H1}), and $\gamma - k\alpha \notin \Phi \cup \{0\}$ for $k > k_0$. This implies $[x_\alpha^{[i]},y_\gamma] = 0$ for $i>k_0$. By Lemma \ref{generallemma} we have

$$x_\alpha^{nk_0}(y_\gamma^{(0)})^n.v^+ = \sum_{
i_1 + \ldots + i_{n+1} = nk_0} {nk_0 \choose i_1  \ldots  i_n \, i_{n+1}} [x_\alpha^{[i_1]},y_\gamma^{(0)}]\cdot \ldots \cdot [x_\alpha^{[i_n]},y_\gamma^{(0)}]x_\alpha^{i_{n+1}}.v^+ $$

\vskip8pt

and as $x$ annihilates $v^+$ this reduces to

$$\sum_{i_1 + \ldots + i_n = nk_0} {nk_0 \choose i_1  \ldots  i_n} [x_\alpha^{[i_1]},y_\gamma^{(0)}]\cdot \ldots \cdot [x_\alpha^{[i_n]},y_\gamma^{(0)}].v^+ \;.
$$

\vskip8pt

By what we have just observed, the corresponding term vanishes if there is one $i_j>k_0$. Therefore, only the term with all $i_j = k_0$ contributes, and this sum is hence equal to

$${nk_0 \choose k_0  \ldots  k_0 } [x_\alpha^{[k_0]},y_\gamma^{(0)}]^n.v^+ = \frac{(nk_0)!}{(k_0!)^n} \; [x_\alpha^{[k_0]},y_\gamma^{(0)}]^n.v^+$$

\vskip8pt

\setcounter{enumi}{0}

\begin{sublemma} $[x_\alpha^{[k_0]},y_\gamma] = k_0! \cdot c \cdot y_{\gamma - k_0\alpha}$ with an integer $c$ which is a unit in $O_K$.
\end{sublemma}

\Pf This is \cite[8.11]{OS2}. We remark that in the proof of this statement the assumption \ref{hyp} is used. \qed

\vskip8pt

We conclude that

\begin{numequation}\label{conclusion}
x_\alpha^{nk_0}(y_\gamma^{(0)})^n.v^+ = (nk_0)! \cdot c^n \cdot (y_{\gamma - k_0 \alpha}^{(0)})^n . v^+ \hskip8pt \mbox{with} \hskip8pt c \in O_K^*\;.
\end{numequation}

If $\gamma - k_0\alpha$ is not in $\Phi_I$ we are done, by the same reasoning as in part (a). Otherwise we necessarily have $\alpha \notin I$. In this case we have by Lemma \ref{generallemma} and Lemma \ref{lemma2} (iii)

$$x_\beta^n(y_\gamma^{(0)})^n.v^+ = n![x_\beta,y_\gamma^{(0)}]^n.v^+ \,.$$

\vskip8pt

Let $-\gamma-r\beta, \ldots, -\gamma+q\beta$ be the  $\beta$-string through $-\gamma$.  
We have $q \ge 2$ and thus $r \le 1$. 
If $r=0$ then $[x_\beta,y_\gamma] = \pm y_\alpha$. If $r=1$ then $[x_\beta,y_\gamma] =\pm 2 y_{\alpha}$. 
But in this case the $\beta$-string through $-\gamma$ consists of four roots, and $\Phi$ must have an irreducible component of type $G_2$. 
By \ref{hyp} we have $p>3$, and thus $2 \in O_K^*$. Therefore, we always have $[x_\beta,y_\gamma] = c y_\alpha$ with $c \in O_K^*$. 
With the same arguments as in part (a) we can thus deduce the claim. \qed

\vskip12pt

\setcounter{enumi}{0}

\begin{lemma}\label{ABCD}
Suppose that none of the irreducible components of $\Phi$ is of type $G_2$. Let $\gamma \in \Phi^+$. Consider a relation

\begin{numequation}\label{relation2}
n \gamma = \nu_1\beta_1 + \ldots + \nu_t\beta_t
\end{numequation}

with non-negative integers $n, \nu_1, \ldots, \nu_t \in \Z_{\ge 0}$. Then $n \le \nu_1 + \ldots + \nu_t$.
\end{lemma}

\Pf This is \cite[8.12]{OS2}. \qed

\vskip8pt

Now we generalize the preceding lemma so as to assure the existence of some $\nu \in \cI_n$ satisfying both conditions $i.$ and $ii.$ of \ref{intro}.

\setcounter{enumi}{0}

\begin{prop}\label{both_conditions} Let $M \in \cO^\frp$, $\frp = \frp_I$, $\gamma \in \Phi^+ \setminus \Phi^+_I$, and $m_0 \in \Z_{\ge 0}$ be as in \ref{intro}. Suppose the residue characteristic of $K$ does not divide any of the non-zero numbers among $\langle \beta, \alpha^\vee\rangle$, $\alpha, \beta \in \Phi$, $\alpha \neq \pm \beta$. Then, for any $n \in \Z_{\ge 0}$ and any expression

\begin{numequation}\label{relation3}
(y_{\gamma}^{(0)})^n.v^+ = \sum_{\nu \in \cI_n} c_\nu (y_1^{(0)})^{\nu_1} \cdot \ldots \cdot (y_t^{(0)})^{\nu_t}.v^+ \;,
\end{numequation}

there is at least one index $\nu \in \cI_n$ such that $\nu_1 + \ldots + \nu_t \ge n$ and $|c_\nu|_K \ge 1$.
\end{prop}

\Pf If $\Phi$ does not have an irreducible component of type $G_2$ then \ref{estimate} and \ref{ABCD} prove the assertion of \ref{both_conditions}. Therefore we assume in the following that $\Phi$ is irreducible of type $G_2$.

The basic idea of the proof is as follows. There is nothing to show if $ht(\gamma) = 1$. Now suppose that $ht(\gamma) > 1$. Then there is $\gamma' \in \Phi^+$ and $k_0 \in \Z_{>0}$ such that $\gamma - k_0\gamma' \in \Phi^+ \setminus \Phi^+_I$ and

$$\frac{1}{(nk_0)!} x_{\gamma'}^{nk_0} \cdot (y_\gamma^{(0)})^n.v^+ = \frac{1}{(k_0!)^n} \cdot [x_{\gamma'}^{[k_0]}, y_\gamma^{(0)}]^n.v^+ = c \cdot (y_{\gamma-k_0\gamma'}^{(0)})^n.v^+ \;,$$

\vskip8pt

with $c \in O_K^*$, cf. \ref{conclusion}. Considering the right hand side of \ref{relation3}, we aim to show that for any $\nu \in \cI_n$ with $\nu_1 + \ldots + \nu_t < n$ the term

\vspace{-0.3cm}
\begin{numequation}\label{vanishing_term}
x_{\gamma'}^{nk_0} \cdot (y_1^{(0)})^{\nu_1} \cdot \ldots \cdot (y_t^{(0)})^{\nu_t}.v^+
\end{numequation}

in

\vspace{-0.3cm}
\begin{numequation}\label{modified_sum}
\frac{1}{(nk_0)!} x_{\gamma'}^{nk_0} \cdot \left(\sum_{\nu \in \cI_n} c_\nu (y_1^{(0)})^{\nu_1} \cdot \ldots \cdot (y_t^{(0)})^{\nu_t}.v^+\right) = \sum_{\nu \in \cI_n} c_\nu \cdot \frac{1}{(nk_0)!} x_{\gamma'}^{nk_0} \cdot (y_1^{(0)})^{\nu_1} \cdot \ldots \cdot (y_t^{(0)})^{\nu_t}.v^+
\end{numequation}

vanishes. This means that the sum on the right hand side of \ref{modified_sum} is equal to

\begin{numequation}\label{modified_sum2}
\sum_{\nu \in \cJ_n} c_\nu \cdot \frac{1}{(nk_0)!}x_{\gamma'}^{nk_0} \cdot (y_1^{(0)})^{\nu_1} \cdot \ldots \cdot (y_t^{(0)})^{\nu_t}.v^+ \;,
\end{numequation}

where $\cJ_n \sub \cI_n$ consists only of those $\nu \in \cI_n$ for which $\nu_1 + \ldots + \nu_t \ge n$. We can then rewrite \ref{modified_sum2} as

$$\sum_{\nu' \in \cI'_n} c'_{\nu'} \cdot (y_1^{(0)})^{\nu_1'} \cdot \ldots \cdot (y_t^{(0)})^{\nu_t'}.v^+ \;,$$

\vskip8pt

where $\cI'_n$ consists of all $\nu' \in \Z_{\ge 0}^t$ such that $\nu_1'\beta_1 + \ldots + \nu_t' \nu_t' = n(\gamma - k_0\gamma')$, and the numbers $c'_{\nu'}$ are linear combinations with integral coefficients of the numbers $c_\nu$, with $\nu \in \cJ_n$. (Recall that the operator $\frac{1}{(nk_0)!}{\rm ad}(x_{\gamma'})^{nk_0}$ preserves $p^{m_0}\frg_\Z'$.) Applying Lemma \ref{estimate}, we find that there is $\nu' \in \cI_n'$ such that $|c'_{\nu'}|_K \ge 1$. Hence there is at least one $\nu \in \cJ_n$ such that $|c_\nu|_K \ge 1$. Thus proving our assertion.

\vskip8pt

We need to show that the term \ref{vanishing_term} actually vanishes if $\nu_1 + \ldots + \nu_t < n$.
This is done exactly as in the remaining parts of the proof of \cite[8.13]{OS2}, which consist of purely Lie algebra arguments and do not involve the field $K$. \qed 

\bibliographystyle{plain}
\bibliography{JoHo}

\def\cprime{$'$}
\begin{thebibliography}{10}

\bibitem{BGR}
S.~Bosch, U.~G{\"u}ntzer, and R.~Remmert.
\newblock {\em Non-{A}rchimedean analysis}, volume 261 of {\em Grundlehren der
  Mathematischen Wissenschaften [Fundamental Principles of Mathematical
  Sciences]}.
\newblock Springer-Verlag, Berlin, 1984.
\newblock A systematic approach to rigid analytic geometry.

\bibitem{Bourbaki_Lie_1_3}
Nicolas Bourbaki.
\newblock {\em Lie groups and {L}ie algebras. {C}hapters 1--3}.
\newblock Elements of Mathematics (Berlin). Springer-Verlag, Berlin, 1998.
\newblock Translated from the French, Reprint of the 1989 English translation.

\bibitem{Car}
Roger~W. Carter.
\newblock {\em Finite groups of {L}ie type}.
\newblock Wiley Classics Library. John Wiley \& Sons Ltd., Chichester, 1993.
\newblock Conjugacy classes and complex characters, Reprint of the 1985
  original, A Wiley-Interscience Publication.

\bibitem{DDMS}
J.~D. Dixon, M.~P.~F. du~Sautoy, A.~Mann, and D.~Segal.
\newblock {\em Analytic pro-{$p$} groups}, volume~61 of {\em Cambridge Studies
  in Advanced Mathematics}.
\newblock Cambridge University Press, Cambridge, second edition, 1999.

\bibitem{F}
Christian~Tobias F{\'e}aux~de Lacroix.
\newblock Einige {R}esultate \"uber die topologischen {D}arstellungen
  {$p$}-adischer {L}iegruppen auf unendlich dimensionalen {V}ektorr\"aumen
  \"uber einem {$p$}-adischen {K}\"orper.
\newblock In {\em Schriftenreihe des {M}athematischen {I}nstituts der
  {U}niversit\"at {M}\"unster. 3. {S}erie, {H}eft 23}, volume~23 of {\em
  Schriftenreihe Math. Inst. Univ. M\"unster 3. Ser.}, pages x+111. Univ.
  M\"unster, M\"unster, 1999.

\bibitem{EGA}
A.~Grothendieck.
\newblock \'{E}l\'ements de g\'eom\'etrie alg\'ebrique. {III}. \'{E}tude
  cohomologique des faisceaux coh\'erents. {I}.
\newblock {\em Inst. Hautes \'Etudes Sci. Publ. Math.}, (11):167, 1961.

\bibitem{H2}
James~E. Humphreys.
\newblock {\em Introduction to {L}ie algebras and representation theory},
  volume~9 of {\em Graduate Texts in Mathematics}.
\newblock Springer-Verlag, New York, 1978.
\newblock Second printing, revised.

\bibitem{H1}
James~E. Humphreys.
\newblock {\em Representations of semisimple {L}ie algebras in the {BGG}
  category {$\mathscr{O}$}}, volume~94 of {\em Graduate Studies in
  Mathematics}.
\newblock American Mathematical Society, Providence, RI, 2008.

\bibitem{K1}
Jan Kohlhaase.
\newblock Invariant distributions on {$p$}-adic analytic groups.
\newblock {\em Duke Math. J.}, 137(1):19--62, 2007.

\bibitem{K2}
Jan Kohlhaase.
\newblock The cohomology of locally analytic representations.
\newblock {\em J. Reine Angew. Math.}, 651:187--240, 2011.

\bibitem{OS1}
Sascha Orlik and Matthias Strauch.
\newblock On the irreducibility of locally analytic principal series
  representations.
\newblock {\em Represent. Theory}, 14:713--746, 2010.

\bibitem{OS2}
Sascha Orlik and Matthias Strauch.
\newblock On {J}ordan-{H}\"older series of some locally analytic
  representations.
\newblock {\em J. Amer. Math. Soc.}, 28(1):99--157, 2015.

\bibitem{Sch}
Tobias Schmidt.
\newblock Auslander regularity of {$p$}-adic distribution algebras.
\newblock {\em Represent. Theory}, 12:37--57, 2008.

\bibitem{S1}
Peter Schneider.
\newblock {\em Nonarchimedean functional analysis}.
\newblock Springer Monographs in Mathematics. Springer-Verlag, Berlin, 2002.

\bibitem{ST3}
Peter Schneider and Jeremy Teitelbaum.
\newblock {$p$}-adic boundary values.
\newblock {\em Ast\'erisque}, (278):51--125, 2002.
\newblock Cohomologies $p$-adiques et applications arithm{\'e}tiques, I.

\bibitem{ST2}
Peter Schneider and Jeremy Teitelbaum.
\newblock Algebras of {$p$}-adic distributions and admissible representations.
\newblock {\em Invent. Math.}, 153(1):145--196, 2003.

\bibitem{Ti}
J.~Tits.
\newblock Reductive groups over local fields.
\newblock In {\em Automorphic forms, representations and {$L$}-functions
  ({P}roc. {S}ympos. {P}ure {M}ath., {O}regon {S}tate {U}niv., {C}orvallis,
  {O}re., 1977), {P}art 1}, Proc. Sympos. Pure Math., XXXIII, pages 29--69.
  Amer. Math. Soc., Providence, R.I., 1979.

\end{thebibliography}

\end{document}